\documentclass[smallcondensed]{svjour3}     
\smartqed  
\usepackage{graphicx}
\usepackage[english]{babel}
\usepackage{cite}
\usepackage[misc]{ifsym}
\usepackage{geometry}
\usepackage{listings} 
\usepackage{algorithmic,algorithm}
\usepackage{multirow}
\usepackage{enumerate}
\usepackage{multirow, array}
\usepackage{datetime}
\usepackage{undertilde}

\usepackage{booktabs}
\usepackage{tabularx}

\usepackage{amscd}
\usepackage[pdftex]{color} 
\usepackage{graphicx}
\usepackage{lscape}
\usepackage{indentfirst}
\usepackage{latexsym}
\usepackage{amsmath, amsfonts, amssymb,mathrsfs}
\usepackage{verbatim}
\usepackage{float}
\usepackage{babel,blindtext}

\usepackage[colorlinks,linkcolor=red,anchorcolor=blue,citecolor=magenta]{hyper ref}

\numberwithin{equation}{section} \numberwithin{table}{section}
\numberwithin{figure}{section}
\numberwithin{algorithm}{section}
\everymath{\displaystyle}

\begin{document}

\title{Fast multilevel solvers for a class of discrete fourth order parabolic problems}

\author{Bin Zheng \and Luoping Chen \and Xiaozhe Hu \and Long Chen \and \\ Ricardo H. Nochetto \and Jinchao Xu   }

\institute{
B. Zheng \at Fundamental \& Computational Sciences, Pacific Northwest National Laboratory, Richland, WA 99352, USA\\
\email{bin.zheng@pnnl.gov} \\
\and
L.P. Chen \at School of Mathematics, Southwest Jiaotong University, Chengdu 611756, China\\
\email{clpchenluoping@163.com}\\ 
\and
X. Hu \at Department of Mathematics, Tufts University, Medford, MA 02155, USA\\
\email{Xiaozhe.Hu@tufts.edu}
\and 
L. Chen \at Department of Mathematics, University of California, Irvine, CA 92697, USA\\
 \email{chenlong@math.uci.edu}\\
\and
R.H. Nochetto \at Department of Mathematics and Institute for Physical Science and Technology, University of Maryland, College Park MD 20742, USA\\ 
\email{rhn@math.umd.edu}
\and 
J. Xu \at Department of Mathematics, Pennsylvania State University, University Park, PA 16802, USA\\
\email{xu@math.psu.edu}
}
\maketitle

\begin{abstract}
In this paper, we study fast iterative solvers for the solution of fourth order parabolic equations discretized by mixed finite element methods. We propose to use consistent mass matrix in the discretization and use lumped mass matrix to construct efficient preconditioners. We provide eigenvalue analysis for the preconditioned system and estimate the convergence rate of the preconditioned GMRes method. Furthermore, we show that these preconditioners only need to be solved inexactly by optimal multigrid algorithms. Our numerical examples indicate that the proposed preconditioners are very efficient and robust with respect to both discretization parameters and diffusion coefficients. We also investigate the performance of multigrid algorithms with either collective smoothers or distributive smoothers when solving the preconditioner systems. 
\end{abstract}

\section{Introduction}
Fourth order parabolic partial differential equations (PDEs) appear in many applications including the lubrication approximation for viscous thin films \cite{king2003fourth},  interface morphological changes in alloy systems during stress corrosion cracking by surface diffusion \cite{asaro1972interface}, image segmentation, noise reduction, inpainting \cite{bertozzi2007inpainting, greer2004h}, and phase separation in binary alloys \cite{blowey1991cahn, elliott1989cahn, cahn2004free}, etc. There have been extensive studies on the numerical methods for solving these fourth order parabolic PDEs, including the finite difference methods \cite{furihata2001stable, sun1995second}, spectral methods \cite{ye2005fourier, he2008stability}, discontinuous Galerkin methods \cite{xia2007local,choo2005discontinuous,feng2007fully}, $C^1$-conforming finite element methods \cite{elliott1987numerical,feng2005numerical},  nonconforming finite element methods\cite{elliott1989nonconforming,zhang2010nonconforming}, and mixed finite element methods \cite{du1991numerical, elliott1989second,elliott1992error,feng2004error,barrett1999finite,BElliot1992Cahn-Hilliard}.   The resulting large sparse linear systems of equations are typically very ill-conditioned which poses great challenges for designing efficient and robust iterative solvers. In addition, the linear systems become indefinite when mixed methods are employed. In this work, we design fast iterative solvers for the fourth order parabolic equations discretized by mixed finite element methods.

Multigrid (MG) methods are among the most efficient and robust solvers for solving the linear systems arising from the discretizations of PDEs. There have been a number of studies on multigrid methods for fourth order parabolic equations \cite{kay2006multigrid,Kim2004,wise2007solving,banas2009multigrid,henn2005multigrid}. It is known that smoothers play a significant role in multigrid algorithms. In particular, for saddle point systems, there are two different types of smoothers, i.e., decoupled or coupled smoothers. For a decoupled point smoother, each sweep consists of relaxation over variables per grid point as well as relaxation over grid points. On the other hand, for a coupled point smoother, variables on each grid point are relaxed simultaneously which corresponds to solving a sequence of local problems. Distributive Gauss-Seidel (DGS) is the first decoupled smoother proposed by Brandt and Dinar \cite{brandt1979multi} for solving Stokes equations. Later, Wittum \cite{wittum1999multigrid} introduced transforming smoothers and combined it with incomplete LU factorization to solve saddle-point systems with applications in Stokes and Navier-Stokes equations. Gaspar, Lisbona, Oosterlee and Wienands \cite{gaspar2004systematic} studied the DGS smoother for solving poroelasticity problems. Recently, Wang and Chen \cite{wang2013multigrid} proposed a distributive Gauss-Seidel smoother based on the least squares commutator for solving Stokes equations and showed multigrid uniform convergence numerically. For coupled smoother, in \cite{Vanka1986block}, Vanka proposed a symmetrically coupled Gauss-Seidel smoother for the Navier-Stokes equations discretized by finite difference schemes on staggered grids. Later, in \cite{Olshanskii2002}, Olshanskii and Reusken introduced an MG solver with collective smothers for the Navier-Stokes equations in rotation form discretized by conforming finite elements. They proved that W-cycle MG method with block Richardson smoother is robust for the discrete problem. Their numerical experiments show that the W-cycle MG method with the damped Jacobi smoother is robust with respect to the mesh size and problem parameters. Collective Jacobi and Gauss-Seidel smoothers have also been studied in multigrid method for elliptic optimal control problems by Lass, Vallejos, Borzi and Douglas \cite{Lass2009}, and Takacs and Zulehner \cite{Takacs2011}. In our work, we investigate the performance of both coupled and decouple smoothers for the discrete fourth order problem in the mixed form. 

In the literature, there are also many studies on the Krylov subspace methods and their preconditioners for solving fourth order parabolic problems. In \cite{bansch2004surface,bansch2005finite}, B{\"a}nsch, Morin and Nochetto solved a discrete fourth order parabolic equation by applying the conjugate gradient method on the two-by-two block system in Schur complement form. In \cite{bansch2011preconditioning}, the same authors proposed symmetric and nonsymmetric preconditioners based on operator splitting. Axelsson, Boyanova, Kronbichler, Neytcheva, Do-Quang, and Wu\cite{Axelsson2011,Boyanova2012,Axelsson2013,Boyanova2014} proposed a block preconditioner by adding an extra small term to the $(2, 2)$ block and then followed by a block LU factorization which results in a preconditioned system when eigenvalues have a uniform bound. Bosch, Kay, Stoll, and Wathen \cite{Bosch2014a} studied block triangular preconditioner for the mixed finite element discretization of a modified Cahn-Hilliard equation for binary image inpainting. In their method, the Schur complement is solved by algebraic multigrid method (AMG) using V-cycle and Chebyshev smoother. Moreover, the convergence rate of their method is independent of the mesh size. Same method has been used by Bosch, Stoll, and Benner in \cite{Bosch2014b} to solve Cahn-Hilliard variational inequalities. In\cite{graser2007preconditioned}, Gr{\"a}ser and Kornhuber introduced a preconditioned Uzawa-type iterative method with multigrid solvers for subproblems to solve Cahn-Hilliard equation with an obstacle potential. Based on the Hermitian and skew-Hermitian splitting (HSS) iteration introduced in \cite{bai2003hermitian}, Benzi and Golub \cite{benzi2004preconditioner} proposed a robust preconditioner for the generalized saddle point problems which exhibit similar two-by-two block structure as the discrete fourth order parabolic problems in mixed form. Later, Bai, Benzi, Chen, and Wang \cite{Bai:2013aa} studied preconditioned modified HSS (PMHSS) iteration methods for the block linear systems. In \cite{Bai2013}, Bai, Chen, and Wang further simplified PMHSS and proposed an additive block diagonal preconditioner for a block linear system of equations arising from finite element discretization of two-phase flow problems and elliptic PDE-constrained optimization problems. 

In this work, we construct preconditioners for discrete fourth order parabolic problems based on the mass lumping technique. Mass lumping technique has been widely used in solving time dependent PDEs by the finite element method \cite{HRZ_1976}. It consists of replacing a consistent mass matrix by a diagonal lumped mass matrix so that its inversion at each time step becomes a simple division. The error estimates for lumped mass approximation have been studied in \cite{Ushijima1977, ushijima1979error, chenthomee1985lumped} which show that the order of accuracy for the discretization is preserved. On the other hand, the loss of solution accuracy associated with mass lumping has been studied by Gresho, Lee, and Sani \cite{gresho1978advection} for advection equations and by Niclasen and Blackburn \cite{Niclasen1995} for incompressible Navier-Stokes equations. It is well known that mass lumping may also induce dispersion errors when solving wave equations, see e.g., \cite{mullen_belytschko_1982,mark_christon_1999}. These studies suggest that it is sometimes advantageous to use consistent mass matrix in the discretization schemes. In the study of fourth order parabolic equations, both consistent mass matrix and lumped mass matrix have been widely used. In this work, we choose consistent mass matrix in the finite element discretization to keep the solution accuracy and utilize lumped mass matrix to design efficient preconditioners so that the cost of inverting consistent mass matrix can be alleviated. We prove that GMRes method with mass lumping preconditioner converges when $\tau \geq Ch^2$ for some constant $C$ depending on the diffusion coefficients ($\tau$, $h$ corresponds to time and spatial discretization parameters, respectively). In a special case when the two diffusion operators $A$ and $B$ only differ by a scaling factor, we are able to prove uniform convergence of GMRes method for the preconditioned system without the constraint $\tau \geq Ch^2$. Furthermore, we show that the preconditioner systems can be solved inexactly by geometric multigrid methods with the two different types of smoother discussed previously. By combining the optimality of multigrid methods with the computational efficiency of the mass lumping technique, we obtain very efficient solvers for the discrete fourth order problems. 

The remainder of the paper is organized as follows. In Section 2, we describe the model problem and the corresponding mixed finite element discretization. Next, in Section 3, we describe the multigrid method and the collective Jacobi/Gauss-Seidel smoothers for our model problem. In Section 4, we construct two mass lumping preconditioners. Multigrid method with decoupled smoother are also introduced in this section to solve the preconditioner system approximately. The spectrum bounds of the preconditioned systems and the convergence property of GMRes method are analyzed in Section 5.  Finally, in Section 6, we present numerical experiments to demonstrate the efficiency and robustness of the proposed solvers.

\section{Model Problem and Discretization}
\subsection{Model problem}
We are interested in solving the following fourth order problem: 
\begin{align}\label{model1}
u-\tau\;\mbox{div}(a\nabla v)&=f \ \;\mbox{in}\;\Omega,\\ \label{model2}
\tau\;\mbox{div}(b\nabla u)+v&=g \ \;\mbox{in}\;\Omega,
\end{align}
with boundary conditions
\begin{align}
\label{bd1}
u=v=0,\quad \mbox{on} \   \Gamma_D, \\
\label{bd2}
\nu\cdot a\nabla v=\nu\cdot b\nabla u=0,\quad\mbox{on} \  \Gamma_N.
\end{align}
where $\tau = \sqrt{\triangle t}$ and $\triangle t$ is the time-step size, $\Omega$ is a bounded polyhedral domain in $\mathbb{R}^d$, $d\geq 1$, $\nu$ is the unit outward normal, $\Gamma_D$, $\Gamma_N$ denote the Dirichlet and Neumann boundary part, respectively. We mainly study the Dirichlet boundary condition in this paper, i.e., $\Gamma_N=\emptyset$. The diffusion coefficients $a(x)$ and $b(x)$ are measurable functions satisfying the following continuity and coercivity conditions
$$
\lambda_{a}(x)|\xi|^2\leq\xi^{T}a(x)\xi\leq\Lambda_{a}(x)|\xi|^2,\quad \lambda_{b}(x)|\xi|^2\leq\xi^{T}b(x)\xi\leq\Lambda_{b}(x)|\xi|^2,\; \forall\xi\in\mathbb{R}^{d}.
$$ 
Eqns (\ref{model1}) and (\ref{model2}) may arise from a time semi-discretization of the fourth order parabolic problem
$$
u_t=\mbox{div}(a\nabla(-\mbox{div}(b\nabla u))).
$$
In this case, $f=u^{\text{old}}$ is the solution at the previous time step, and $g=0$.

As an example, consider the following Cahn-Hilliard equations that model the phase separation and coarsening dynamics in a binary alloy:
\begin{equation}
u_t-\mbox{div} (M(u)\nabla (-\epsilon\Delta u+F'(u))) = 0,\;\mbox{in}\;\Omega\times (0, T),
\label{Cahn-Hilliard-fourth-order-form}
\end{equation}
where $\Omega\subset \mathbb{R}^3$ is a bounded domain, $u$ represents the relative concentration of one component in a binary mixture, $M(u)$ is the degenerate mobility, which restricts diffusion of both components to the interfacial region, e.g., $M(u) = u(1-u)$ \cite{barrett1999finite}, $F'(u)$ is the derivative of a double well potential $F(u)$, a typical choice is
$$
F(u) = \frac{1}{4} (u^2-1)^2.
$$
Introducing $v$, defined by $v = \tau(-\epsilon \Delta u + F'(u))$(the chemical potential multiplied by $\tau$), after semi-implicit time discretization, we obtain a splitting of (\ref{Cahn-Hilliard-fourth-order-form}) into a coupled system of second order equations \cite{elliott1989second, elliott1992error,feng2004error}
\begin{align*}
u - \tau\mbox{div}(M(u)\nabla v)& = f,\;\mbox{in}\;\Omega,\\ 
\tau(-\epsilon \Delta u + F'(u)) - v& = 0,\;\mbox{in}\;\Omega.
\end{align*}
Denote $a=M(u)$, $b=\epsilon$, and assume $F'(u)=0$, we can get (\ref{model1}) and (\ref{model2}), which corresponds to the linearization of the nonlinear Cahn-Hilliard equations.

The weak formulation of (\ref{model1})-(\ref{model2}) is: find $u,v\in\mathbb{V}$ such that
\begin{equation}
\begin{array}{ccc}
(u,\phi)+\tau(a\nabla v, \nabla \phi) & = &(f,\phi),\quad \phi\in\mathbb{V}\\
-\tau(b\nabla u, \nabla\psi) + (v,\psi) & = & (g,\psi),\quad \psi\in\mathbb{V},
\end{array}
\label{weak-form}
\end{equation}
where $\mathbb{V}$ is the subspace of $H^1(\Omega)$ associated with the boundary condition (\ref{bd1}). The well-posedness of (\ref{weak-form}) follows from the Lax-Milgram lemma \cite{quarteroni1980mixed}.

\subsection{Finite element discretization}
Let $\mathcal{T}$ be a shape-regular triangulation of $\Omega$, $\mathbb{V}_{\mathcal{T}}$ be the piecewise linear finite element space over $\mathcal{T}$ satisfying the homogeneous boundary condition (\ref{bd1}), and $N_h = {\rm dim(\mathbb{V}_\mathcal{T})}$. The discrete problem for the PDE system (\ref{model1})-(\ref{model2}) is: find $u_h, v_h\in\mathbb{V}_{\mathcal{T}}$, such that $\forall \phi, \psi\in\mathbb{V}_{\mathcal{T}}$
$$
\begin{array}{ccc}
(u_h,\phi)+\tau(a\nabla v_h, \nabla \phi) & = &(f,\phi),\\
-\tau(b\nabla u_h, \nabla\psi) + (v_h,\psi) & = & (g,\psi).
\end{array}
$$
In matrix form, we have
\begin{equation}
\left(
\begin{array}{rc}
\tau A & M\\
M & -\tau B
\end{array}
\right)\left(
\begin{array}{c}
\utilde{v}\\
\utilde{u}
\end{array}
\right) = 
\left(
\begin{array}{r}
\utilde{f}\\
\utilde{g}
\end{array}
\right),\;\;\text{or}\;\;\mathcal{A}
\left(
\begin{array}{c}
\utilde{v}\\
\utilde{u}
\end{array}
\right) = 
\left(
\begin{array}{r}
\utilde{f}\\
\utilde{g}
\end{array}
\right),
\label{2by2-block-matrix-form}
\end{equation}
where $\mathcal{A}$ is a matrix of size $2N_h\times 2N_h$, $M$ is the mass matrix, and $A$, $B$ are the stiffness matrices defined by
$$
M_{i,j}=(\phi_i,\phi_j),\quad A_{i,j}=(a\nabla\phi_i, \nabla\phi_j), \quad B_{i,j}=(b\nabla\phi_i,\nabla\phi_j).
$$
 Eliminating $\utilde{v}$ in (\ref{2by2-block-matrix-form}), we get the following Schur complement equation
\begin{equation}
(M+\tau^2AM^{-1}B)\utilde{u}=\utilde{f}-\tau AM^{-1}\utilde{g}.
\label{schur-complement-form}
\end{equation}
In this paper, we develop efficient preconditioners for the two-by-two block linear systems (\ref{2by2-block-matrix-form}); however, we will solve the Schur complement for preconditioners. In order to avoid inverting mass matrix, we employ the mass lumping technique to construct preconditioners.

\section{Multigrid with Collective Smoother}
\label{sec:MGsolver}
A multigrid algorithm typically consists of three major components: the smoother (or relaxation scheme), the coarse grid operator, and the grid transfer operators (interpolation/restriction operators). It is well known that the efficiency of a multigrid method crucially depends on the choice of the smoother. In particular, for the block system (\ref{2by2-block-matrix-form}), we observe numerically that multigrid method with point-wise Gauss-Seidel smoother does not converge uniformly. We construct a block Jacobi or Gauss-Seidel smoother that collects the degrees of freedom corresponding to variables $u$ and $v$ for each grid point. In other words, each block corresponds to a $2\times 2$ matrix.

To describe these collective smoothers, we first consider the following matrix splitting, i.e., 
$$
\mathcal{A}=\mathcal{A}_L +\mathcal{A}_D + \mathcal{A}_L^T,\;\;
\mathcal{A}_L =
\left(
\begin{array}{cc}
\tau L_A & L_M\\
L_M &-\tau L_B
\end{array}
\right),
\;\;
\mathcal{A}_D =
\left(
\begin{array}{cc}
\tau D_A & D_M\\
D_M &-\tau D_B
\end{array}
\right),
$$
where $L_A , L_B, L_M$ are strictly lower triangular parts of $A$, $B$, and $M$, and $D_A , D_B, D_M$ are their diagonal parts. The collective damped Jacobi relaxation can be represented as
\begin{align}\label{collective-Jacobi-smoother}
\begin{pmatrix}
\utilde{v}^{k+1}\\
\utilde{u}^{k+1}
\end{pmatrix}
=\begin{pmatrix}
 \utilde{v}^{k}\\
\utilde{u}^{k}
\end{pmatrix}
+
\vartheta\mathcal{A}_D^{-1}
\left[\begin{pmatrix}
 \utilde{f}\\
 \utilde{g}
\end{pmatrix}
-\mathcal{A}
\begin{pmatrix}
 \utilde{v}^{k}\\
\utilde{u}^{k}
\end{pmatrix}\right]
\end{align}

The collective Gauss-Seidel relaxation can be represented as
\begin{align}\label{collective-Gauss-Seidel-smoother}
\begin{pmatrix}
\utilde{v}^{k+1}\\
\utilde{u}^{k+1}
\end{pmatrix}
=\begin{pmatrix}
 \utilde{v}^{k}\\
\utilde{u}^{k}
\end{pmatrix}
+
(\mathcal{A}_L+\mathcal{A}_D)^{-1}
\left[\begin{pmatrix}
 \utilde{f}\\
 \utilde{g}
\end{pmatrix}
-\mathcal{A}
\begin{pmatrix}
 \utilde{v}^{k}\\
\utilde{u}^{k}
\end{pmatrix}\right]
\end{align}
More practically, one can rearrange variables $\utilde{u}$ and $\utilde{v}$ in $\utilde{w}$ so that $\mathcal{A}_L+\mathcal{A}_D$ corresponds to a lower block triangular matrix. In fact, let $\utilde{w}=(\utilde{w}_1, \utilde{w}_2, \dots, \utilde{w}_{N_h})^T$ with each entry $\utilde{w}_i=(v_i, u_i)$ corresponds to a pair of variables. Then, each relaxation sweep of (\ref{collective-Jacobi-smoother}) consists of solving a sequence of small systems
\begin{align*}
\utilde{w}_i^{k+1} = \utilde{w}^{k}_i +\vartheta \mathscr{A}^{-1}_{ii}\left(F_i-\sum_{j=1}^{N_h}\mathscr{A}_{ij}\utilde{w}^{k}_j\right),
\end{align*}
and each sweep of (\ref{collective-Gauss-Seidel-smoother}) consists of solving
\begin{align*}
\utilde{w}_i^{k+1} = \utilde{w}^{k}_i +\mathscr{A}^{-1}_{ii}\left(F_i-\sum_{j=1}^{i-1}\mathscr{A}_{ij}\utilde{w}^{k+1}_j-\sum_{j=i}^{N_h}\mathscr{A}_{ij}\utilde{w}^{k}_j\right),
\end{align*} 
where
$$
\mathscr{A}_{ij} = 
\begin{pmatrix}
\tau A _{ij}& M_{ij}\\
M_{ij} & -\tau B_{ij}
\end{pmatrix},\;\;
F_i = 
\begin{pmatrix}
f_i \\
g_i
\end{pmatrix},\quad i, j = 1, 2, \cdots, N_h.
$$
Note that $\mathscr{A}_{ii}$ is invertible since $\text{det} (\mathscr{A}_{ii}) =  -\tau^2 A_{ii} B_{ii} - M_{ii}^2 \neq 0$.

It is clear from the above form that for collective damped Jacobi relaxation, these small system can be solved in parallel, and for collective Gauss-Seidel relaxation, they are solved successively. Numerical experiments in Section \ref{sec:Numerical experiments} indicate that by relaxing both variables $u_i,v_i$ corresponding to the same grid point $i$ collectively, geometric multigrid V-cycle converges uniformly with respect to both $h$ and $\tau$.

The collective Gauss-Seidel smoother has been studied by Lass, Vallejos, Borzi and Douglas in \cite{Lass2009} for solving elliptic optimal control problems. It is shown in \cite{Lass2009} that the convergence rate of multigrid method with collective Gauss-Seidel smoother is independent of $h$. The robustness with respect to $\tau$ is, however, hard to prove theoretically. Takacs and Zulehner \cite{Takacs2011} also investigated multigrid algorithm with several collective smoothers for optimal control problems and proved the convergence of the W-cycle multigrid with collective Richardson smoother. For convergence proof of multigrid methods using collective smoothers applied to saddle point systems, we refer to the work by Sch\"{o}berl \cite{Schoberl:1999aa} and by Chen \cite{Chen:aa,Chen:2015aa}.

\section{Mass Lumping Preconditioners}
\label{sec:preconditioner}
In this section, we study preconditioners for the two-by-two  block linear systems (\ref{2by2-block-matrix-form}) for use with GMRes method. When the meshsize and time-step size are small, these systems are generally very ill-conditioned, especially when the diffusion coefficients $a$ and $b$ degenerate. Hence, efficient preconditioners are necessary in order to speed up the convergence of GMRes method. In \cite{bansch2011preconditioning}, B\"{a}nsch, Morin, and Nochetto proposed symmetric and non-symmetric preconditioners that work well for the Schur complement system (\ref{schur-complement-form}). However, the convergence rates of these methods are not uniform with respect to $h$ or $\tau$. Besides, the performance of those methods deteriorate for degenerate problems \cite{bansch2011preconditioning}. 

In the following, we design two preconditioners based on the mass lumping technique and geometric multigrid method for the system (\ref{2by2-block-matrix-form}). The main focus is on the efficiency and robustness of the proposed solvers. Numerical experiments in Section \ref{sec:Numerical experiments} indicate that GMRes method preconditioned by the mass lumping preconditioners (solved inexactly) converges uniformly with respect to the discretization parameters and is also robust for problems with degenerate diffusion coefficients. 

\subsection{Preconditioner with two lumped mass matrices}
\label{mass-lumping-preconditioner}
The mass lumping technique has been widely used in the finite element computations, especially for time-dependent problems as it avoids inverting a full mass matrix $M$ at each time step. The lumped-mass matrix $\bar{M}$ is a diagonal matrix with diagonal elements equal to the row sums of $M$. By using the diagonal matrix $\bar{M}$, the computational cost for the preconditioner is reduced significantly. 

The mass lumping preconditioner $\mathcal{B}$ for the block system $\mathcal{A}$ is defined by
\begin{equation}
\mathcal{B} = \begin{pmatrix}
 \tau A & \bar M\\
 \bar M & -\tau B
\end{pmatrix}.
\label{mass-lumping-preconditioner-block-system}
\end{equation}
GMRes method is then applied to solve the preconditioned linear system
 \begin{equation}
\begin{pmatrix}
\tau A & {M}\\
 {M}& -\tau B 
\end{pmatrix}
\begin{pmatrix}
\tau A & \bar{M}\\
 \bar{M}& -\tau B 
\end{pmatrix}^{-1}
\left(
\begin{array}{c}
\utilde{p}\\
\utilde{q}
\end{array}
\right)=
\left(\begin{array}{c}
\utilde{f}\\
\utilde{g}
\end{array}\right).
\end{equation}\label{B-preconditioner}

\subsection{Preconditioner with one lumped mass matrix}
\label{modified-preconditioner}
We can also use the following preconditioner $\tilde{\mathcal{B}}$ with one lumped mass matrix for solving the block system (\ref{2by2-block-matrix-form}), i.e.
\begin{equation}
\tilde{\mathcal{B}}= \begin{pmatrix}
\tau A & {M}\\
 \bar{M}& -\tau B 
\end{pmatrix}.
\label{modified-mass-lumping-preconditioner-block}
\end{equation}
The matrix $\tilde{\mathcal{B}}$ is nonsymmetric and is a better approximation to $\mathcal{A}$ compared with $\mathcal{B}$. For $\tilde{\mathcal{B}}$, we still have a block factorization which avoids inverting mass matrix $M$.  Numerical experiments of Section \ref{subsec:6.2} indicate that the performance of the two preconditioners $\tilde{\mathcal{B}}$ and $\mathcal{B}$ used with GMRes method are similar when solved inexactly by geometric multigrid V-cycle when $\tau\geq Ch^2$. However, when $\tau$ is very small, $\tilde{\mathcal{B}}$ performs better than $\mathcal{B}$. 

\subsection{Multigrid for Preconditioners}
\label{subsection:mg_dgs}
The preconditioner systems only need to be solved approximately. We can use geometric multigrid method with the coupled smoothers described in Section \ref{sec:MGsolver}. In the following, we construct a decoupled smoother following the idea of the distributive Gauss-Seidel relaxation (DGS) which is suitable for the model problem.

DGS is a decoupled smoother introduced by Brandt in \cite{brandt1979multi} for solving Stokes equations. The main idea of DGS is to apply standard Gauss-Seidel relaxation on decoupled equations using transformed variables. Let us consider the mass-lumping preconditioner system $\tilde{\mathcal{B}}$ (a similar scheme can be derived for $\mathcal{B}$),
\begin{equation}
\tilde{\mathcal{B}}
\left(
\begin{array}{c}
\utilde{v}\\
\utilde{u}
\end{array}
\right)
=\left(
\begin{array}{c}
\utilde{p}\\
\utilde{q}
\end{array}
\right).
\label{modified-preconditioner-system}
\end{equation}

We introduce the following change of variables
\begin{equation}
\left(
\begin{array}{c}
\utilde{v}\\
\utilde{u}
\end{array}
\right)
=\mathcal{P}
\left(
\begin{array}{c}
\utilde{x}\\
\utilde{y}
\end{array}
\right),\;\;\text{where}\;\;
\mathcal{P}
=\left(
\begin{array}{cc}
\tau\bar{M}^{-1}B & 0\\
I & I
\end{array}
\right)
\label{change-of-variables}
\end{equation}
is called the distribution matrix. Right preconditioning $\tilde{\mathcal{B}}$ by $\mathcal{P}$ results in an upper block triangular matrix
\begin{align*}
\tilde{\mathcal{B}}\mathcal{P}
=
\left(
\begin{array}{cc}
M+\tau^2A\bar M^{-1}B & M\\
 0 & -\tau B
\end{array}
\right).
\end{align*}

We construct a decoupled smoother for preconditioner (\ref{modified-preconditioner-system}) by solving
\begin{equation}\label{DGS-relaxation}
\begin{pmatrix}
\utilde{v}^{k+1}\\
\utilde{u}^{k+1}
\end{pmatrix}
=
\begin{pmatrix}
\utilde{v}^{k}\\
\utilde{u}^{k}
\end{pmatrix}+
\mathcal{P}(\tilde{\mathcal{B}}\mathcal{P})^{-1}
\left[
\begin{pmatrix}
\utilde{p}\\
\utilde{q}
\end{pmatrix}
-
\tilde{\mathcal{B}}
\begin{pmatrix}
\utilde{v}^k\\
\utilde{u}^k
\end{pmatrix}
\right]
\end{equation}
More precisely, we have the following algorithm
\vspace{0.1cm}

\par
\vspace{0.1cm}
\framebox{
\centering
\parbox{12cm}{
\begin{enumerate}
\item[1.] Form the residual: 
$$
\begin{pmatrix}
 \utilde{r}_v\\
\utilde{r}_u
\end{pmatrix}
= \left[\begin{pmatrix}
\utilde{p}\\
\utilde{q}
\end{pmatrix}
-{\tilde{\mathcal{B}}}
\begin{pmatrix}
 \utilde{v}^{k}\\
 \utilde{u}^{k}
\end{pmatrix}\right]
$$
\item[2.] Apply standard Gauss-Seidel relaxation to solve the error equation
$$
-\tau B\utilde{e}_y=\utilde{r}_u
$$
and use the damped Jacobi relaxation for
\begin{equation}
(M+\tau^2A\bar M^{-1}B)\utilde{e}_x = \utilde{r}_v-M\utilde{e}_y.
\label{fourth-order-in-DGS}
\end{equation}
Then, recover $\utilde{e}_u,\utilde{e}_v$ from (\ref{change-of-variables}), i.e., 
\begin{align*}
\utilde{e}_v &= \tau\bar M^{-1}B\utilde{e}_x;\\
\utilde{e}_u &=\utilde{e}_x+\utilde{e}_y
\end{align*}
\item[3.] Update the solution
\begin{equation*}
\begin{array}{cc}
\utilde{v}^{k+1} = \utilde{v}^{k}+\utilde{e}_v\\
\utilde{u}^{k+1} = \utilde{u}^{k}+ \utilde{e}_u
\end{array}
\end{equation*}
\end{enumerate}
}
}
\vspace{0.3cm}

We can use multigrid method with the above decoupled smoother to solve the preconditioner system (\ref{modified-preconditioner-system}).

\begin{remark}
By using the damped Jacobi to solve (\ref{fourth-order-in-DGS}), we can avoid matrix multiplication. In fact, we only need to calculate the diagonal part of $M+\tau^2A\bar{M}^{-1}B$. This requires $\mathcal{O}(N_h)$ operations because $\bar{M}$ is a diagonal matrix.
\end{remark}

\section{Convergence Analysis of Preconditioned GMRes Method}
\label{sec:convergence}
In this section, we analyze the convergence of the GMRes method preconditioned by the two 
preconditioners introduced in Section \ref{sec:preconditioner} for the block system (\ref{2by2-block-matrix-form}). We show that
 the preconditioned GMRes method converges when $\tau\geq Ch^2$ for some constant $C$ depending on the diffusion coefficients. Numerical results in Section \ref{sec:Numerical experiments} indicate that the preconditioned GMRes method converges uniformly for any $h$ and $\tau$.  

Let $R=\mathcal{A}\mathcal{B}^{-1}$ (or $\mathcal{A}\tilde{\mathcal{B}}^{-1}$) with $\mathcal{B}$ (or $\tilde{\mathcal{B}}$) being a right preconditioner of $\mathcal{A}$. To solve the preconditioned system $Rv=b$, the GMRes method starts from an initial iterate $v_0$ and produces a sequence of iterates
$v_k$ and residuals $r_k := b-Rv_k$, such that  $r_k = p_k(R) r_0$ for some polynomial $p_k\in\mathcal{P}_k$, and

\begin{equation}\label{convergrate}
\|r_k\|_2 = \min_{\substack{p\in \mathcal{P}_k\\ p(0)=1}}\|p(R) r_0\|_2,
\end{equation}
where $\mathcal{P}_k$ is the space of polynomials of degree $k$ or less and $\|\cdot\|_2$ is the Euclidean norm. As a consequence, the convergence rate of the GMRes method can be estimated by

\begin{equation}\label{rate4gmres}
\frac{\|r_k\|_2}{\|r_0\|_2}\leq \min_{\substack{p\in \mathcal{P}_k\\ p(0)=1}}\|p(R)\|_2 .
\end{equation}

Since $R$ may not be a normal matrix, we follow the approach given in \cite{bansch2011preconditioning} using the concept of $\epsilon$-pseudospectrum \cite{trefethen2005spectra}. Given $\epsilon>0$, denote by $\sigma_{\epsilon}(R)$ the set of $\epsilon$-eigenvalues of $R$, namely those $z\in\mathbb{C}$ that are eigenvalues of some matrix $R+E$ with $\|E\|\leq\epsilon$. 
We first quote the following two results.

\begin{lemma}[Pseudospectrum estimate \cite{trefethen2005spectra}]{\label{max-poly}} Let $\Sigma_\epsilon$ be a union of closed curves enclosing the $\epsilon-$pseudospectrum $\sigma_{\epsilon}(R)$ of $R$. Then for any polynomial $p_k$ of degree $k$ with $p_k(0)=1$ we have
\begin{equation}\label{max-poly-f}
\max_{z\in\sigma(R)}|p_k(z)|\leq\|p_k(R)\|\leq\frac{L_{\epsilon}}{2\pi\epsilon}\max_{z\in\Sigma_{\epsilon}(R)}|p_k(z)|,
\end{equation}
where $L_{\epsilon}$ is the arclength of $\Sigma_{\epsilon}$.
\end{lemma}

\begin{lemma}[Bound on the $\epsilon$-pseudospectrum \cite{bansch2011preconditioning}] \label{p-bound}If $R$ is a square matrix of order $n$ and $0<\epsilon\leq 1$, then, 
$$
\sigma_{\epsilon}(R)\subset\cup_{\lambda\in\sigma(R)}B(\lambda, C_R\epsilon^{\frac{1}{m}}),
$$
where $C_R:=n(1+\sqrt{n-p})\kappa(V)$, with $\kappa(V)$ the condition number of $V$ and $V$ is a nonsingular matrix transforming $R$ into its Jordan canonical form $J$, i.e. $V^{-1}RV=J$, $p$ is the number of Jordan blocks, and $m$ is the size of the largest Jordan block of $R$.
\end{lemma}

In order to estimate the error caused by mass lumping, we introduce the following operators \cite{chenthomee1985lumped}. Let $z_j (1\leq j\leq 3)$ be the vertices of a triangle $K \in\mathcal{T}$, consider the following quadrature formula
\begin{align*}
Q_{K,h}(f) &=\sum_{j=1}^3(1,\phi_j|_{K})f(z_j)= \frac{1}{3}|K|\sum_{j=1}^3f(z_j) \\
(\phi_i,\phi_j)_h &= \sum_{K\in\mathcal{T}}Q_{K,h}(\phi_i\phi_j)=\int_{\Omega}\pi_h(\phi_i\phi_j)dx
\end{align*}
where $\pi_h: C^0(\bar{\Omega})\rightarrow \mathbb{V}_{\mathcal{T}}$ is the standard nodal interpolation operator. Then, the mass lumping procedure can be interpreted as
$$
(\phi_j,\phi_j)_h = Q_{K,h}(\phi_j\phi_j) = \sum_{k=1}^{N_h}(\phi_j,\phi_k).
$$
The following results are useful for later analysis.
\begin{lemma}[quadrature error \cite{chenthomee1985lumped}] \label{lumpmass-mass-error}
Let $u,v\in \mathbb{V}_{\mathcal{T}}\subset H_0^1$ and 
$
(u,v)_h
$
be the lumped inner product. Then, 
\begin{align}
|(u,v)-(u,v)_h|\leq C_l h^2|u|_1|v|_1,
\label{property3-M}
\end{align}
where $C_l$ is a constant independent of $h$.
\end{lemma}

\begin{lemma}[Norm equivalence \cite{elliott1989second}] 
Let $M$ be the mass matrix and $\bar{M}$ be its lumped version. Then, for any $u=\textstyle{\sum}_j u_j\phi_j\in\mathbb{V}_\mathcal{T}$ we have
\begin{equation}\label{property1-M}
C_1(\bar{M}\utilde u,\utilde u)\leq(M\utilde u,\utilde u)\leq (\bar{M}\utilde u,\utilde u),
\end{equation}
where $\utilde{u} = (u_1,u_2,\cdots,u_{N_h})^{T}$ and $C_1\in(0,1)$ is a constant independent of $h$. 
\end{lemma}
\begin{proof}
We first show that $({M}\utilde u,\utilde u)\leq (\bar{M}\utilde u,\utilde u)$. Let $\delta M = M - \bar M$, then
$$
(\delta M)_{ij}<0, \ i=j; \quad(\delta M)_{ij} \geq 0, \ i\neq j; \ \ {\rm and} \ \  (\delta M)_{ii}=- \sum_{j\neq i}(\delta M)_{ij}.
$$
By direct calculations, we have
\begin{eqnarray*}
{\utilde{u}}^T\delta M\utilde{u} & = &\sum_{i=1}^{N_h}\sum_{j=1}^{N_h}(\delta M)_{ij} u_i u_j\\
& =& \sum_{i=1}^{N_h}\left({u}_i\sum_{j<i}(\delta M)_{ij}({u}_j-{u}_i)+{u}_i\sum_{j>i}(\delta M)_{ij}({u}_j-{u}_i)\right)\\
& = & \sum_{i=1}^{N_h}\sum_{j>i}-(\delta M)_{ij}({u}_j-{u}_i)^2\leq 0.
\end{eqnarray*}

By Lemma \ref{lumpmass-mass-error} and the inverse inequality, we get
\begin{align*}
(u,u)_h-(u,u)&\leq|(u,u)-(u,u)_h|\leq C_l h^2|u|_1^2\leq C\|u\|_0^2= C(u,u).
\end{align*}
Hence, $ (u,u)_h\leq(1+ C)(u,u)$. Equivalently,
$$
\frac{1}{1+ C}(\bar M\utilde u,\utilde u)\leq (M\utilde u,\utilde u)
$$
which implies that the left inequality holds with $C_1={1}/({1+ C})$ for some positive constant $C$.
\end{proof}

We also need the following estimates for the eigenvalues of the stiffness matrices and mass matrix \cite{fried1973bounds}, \cite{brenner2008book}.
\begin{lemma}[Eigenvalue estimates] \label{spectral4ABM}
Let $A, B$ be the stiffness matrices corresponding to diffusion coefficients $a(x)$ and $b(x)$, respectively, and let $M$ be the mass matrix, then, 
\begin{align}
\lambda_{\max}(A)\leq C_A,\quad \lambda_{\max}(B)\leq C_B,\quad C_Mh^2\leq\lambda_{\min}(M)
\end{align}
where $C_A, C_B$ depend on the continuity assumption of the bilinear forms and the constant in the inverse inequality. 
\end{lemma}

\subsection{Eigenvalue analysis for $\mathcal{A}\mathcal{B}^{-1}$ with constraint $\tau \geq Ch^2$ }
 
Recall
\begin{align*}
 \mathcal{B} = 
 \begin{pmatrix}
 \tau A & \bar M\\
\bar M & -\tau B
\end{pmatrix},
\quad
 \mathcal{A} = 
 \begin{pmatrix}
 \tau A &  M\\
 M & -\tau B
\end{pmatrix}.
 \end{align*}
 We have the following spectrum bound for the preconditioned system $\mathcal{B}^{-1}\mathcal{A}$.
 
 \begin{theorem}[Spectral bound for $\mathcal{B}^{-1}\mathcal{A}$] \label{bdyeig}
Let $h$ denote the meshsize and $\tau$ denote the square root of time-step size. Then, the spectral radius of $\mathcal{B}^{-1}\mathcal{A}$ satisfies
$$
\rho(\mathcal{B}^{-1}\mathcal{A})<2,
$$ 
if ${\tau}\geq Ch^2$ for some positive constant $C$ independent of $h$ and $\tau$.
\end{theorem}
\begin{proof}
Let $\lambda\in\mathbb{C},\  (\utilde v,\utilde u)^{T}\in \mathbb{C}^{2 N_h}$ be a pair of eigenvalue and eigenvector of $\mathcal{B}^{-1}\mathcal{A}$, i.e.
\begin{align*}
\begin{pmatrix}
 \tau A&\bar{M} \\
 \bar{M} & -\tau B
\end{pmatrix}^{-1}
\begin{pmatrix}
 \tau A&{M} \\
 {M} & -\tau B
\end{pmatrix}
\begin{pmatrix}
\utilde v\\
\utilde u
\end{pmatrix}
=
\lambda
\begin{pmatrix}
 \utilde v\\
\utilde u
\end{pmatrix}.
\end{align*}

Equivalently,

\begin{equation}
\begin{pmatrix}
 \tau A&{M} \\
 {M} & -\tau B
\end{pmatrix}
\begin{pmatrix}
\utilde v\\
\utilde u
\end{pmatrix}
=
\lambda
\begin{pmatrix}
 \tau A&\bar M \\
 \bar M & -\tau B
\end{pmatrix}
\begin{pmatrix}
\utilde v\\
\utilde u
\end{pmatrix}.
\label{mass-lumping-preconditioner-e-analysis}
\end{equation}

Taking the inner product of the equation (\ref{mass-lumping-preconditioner-e-analysis}) with $(\utilde v^T,-\utilde u^T)$, we obtain
\begin{equation}
[\tau(A\utilde v,\utilde v)+\tau(B\utilde u,\utilde u)+2\imath{\rm{Im}}(M\utilde v,\utilde u)]=\lambda[\tau(A\utilde v,\utilde v)+\tau(B\utilde u,\utilde u)+2\imath{\rm{Im}}(\bar M\utilde v,\utilde u)].
\end{equation}
Then, we have

\begin{equation}
2\imath{\rm Im}(\delta M\utilde v,\utilde u)=(\lambda-1)[\tau(A\utilde v,\utilde v)+\tau(B\utilde u,\utilde u)+2\imath{\rm{Im}}(\bar M\utilde v,\utilde u)].
\label{e-equality}
\end{equation}
Since $A$ and $B$ are symmetric positive definite matrices, $\|\utilde v\|^2_A=(A\utilde v, \utilde v)$ and $\|\utilde u\|^2_B=(B\utilde u, \utilde u)$ are nonnegative real numbers. By taking the modulus of (\ref{e-equality}) we get
\begin{equation*}
4|{\rm Im}(\delta M\utilde v, \utilde u)|^2=|\lambda-1|^2(\alpha^2+4|{\rm Im}(\bar{M}\utilde v,\utilde u)|^2),
\end{equation*}
where $\alpha=\tau(\|\utilde v\|_{A}^2+\|\utilde u\|_B^2)$. Hence,

\begin{equation}\label{lambdaval}
|\lambda-1|^2=\frac{4|{\rm Im}(\delta M\utilde v, \utilde u)|^2}{\alpha^2+4|{\rm{Im}}(\bar{M}\utilde v,\utilde u)|^2}.
\end{equation}

Let $v={\textstyle\sum}_i v_i\phi_i$, $u={\textstyle\sum}_i u_i\phi_i$, we get
\begin{eqnarray*}
|(\delta M\utilde v,\utilde u)|&\leq& C_l h^2|v|_1|u|_1 \quad \text{(by (\ref{property3-M}))}\\
                     &\leq& C_l C_{a}C_{b}h^2\|\utilde v\|_A\|\utilde u\|_B \quad (C_a=\min_{x}\lambda_a(x),\ C_b=\min_{x}\lambda_b(x))\\
                     &\leq& C_l C_{a}C_{b}h^2\frac{\|\utilde v\|_A^2+\|\utilde u\|^2_B}{2} \quad
                     \text{(by Young's inequality)}\\
                     & = & \frac{C_l C_{a}C_{b}h^2}{2\tau}\tau\left(\|\utilde v\|_A^2+\|\utilde u\|^2_B\right),
\end{eqnarray*}

Hence,
\begin{align*}
4{\rm Im}|(\delta M\utilde v,\utilde u)|^2 \leq\frac{(C_lC_aC_b)^2h^4}{\tau^2}\tau^2(\|\utilde v\|_A^2+\|\utilde u\|_B^2)^2 = \frac{(C_lC_aC_b)^2h^4}{\tau^2}\alpha^2.              
\end{align*}
Let $C = 2C_lC_aC_b$, when $\tau\geq Ch^2$, we get $\rho(\mathcal{B}^{-1}\mathcal{A})<2$.

\end{proof}
Using the same proof of Corollary 5.11 in \cite{bansch2011preconditioning} with $\epsilon_0={1-\rho(\mathcal{B}^{-1}\mathcal{A})/2}$, we have 
\begin{corollary}[Convergence rate of GMRes for $\mathcal{A}\mathcal{B}^{-1}$] \label{GMRes_lump}
For the preconditioned system $\mathcal{A}\mathcal{B}^{-1}$, GMRes method converges with an asymptotic linear convergence rate bounded by
$$
\theta = \frac{\rho(\mathcal{B}^{-1}\mathcal{A})}{2},
$$
if $\tau\geq Ch^2$.
Moreover, 
$$
\frac{\|r_k\|_2}{\|r_0\|_2}\leq C_0\theta^k,
$$
with $C_0=2^{m-1}C_R^m(1+\rho(E)){\rm dim}\mathbb{V}_\mathcal{T}/(1-\rho(E))^m$, $C_R, \ m$ are the parameters defined in the Lemma \ref{p-bound} and $\mathbb{V}_\mathcal{T}$ is the finite element space.
\end{corollary}

Similarly, we have the following result.
\begin{corollary}[Convergence rate of GMRes for $\mathcal{A}\tilde{\mathcal{B}}^{-1}$] \label{convergence-for-tildeBA}
The spectral radius of $\tilde{\mathcal{B}}^{-1}\mathcal{A}$ satisfies
$$
\rho(\tilde{\mathcal{B}}^{-1}\mathcal{A})<2,
$$
if ${\tau}\geq Ch^2$ for some constant $C$ independent with $h$ and $\tau$. For the preconditioned system $\mathcal{A}\tilde{\mathcal{B}}^{-1}$, GMRes method converges with an approximate linear convergence rate bounded by 
$$
\tilde{\theta} = \frac{\rho(\tilde{\mathcal{B}}^{-1}\mathcal{A})}{2}.
$$
\end{corollary}

\subsection{Analysis without the constraint $\tau\geq Ch^2$} As we will see in Section \ref{sec:Numerical experiments}, numerical results show uniform convergence rate of GMRes method for the preconditioned system irrespective of the relation between $h$ and $\tau$. A proof of this result appears elusive due to the fact that the Schur complement is nonsymmetric. In this subsection, we give a proof in this direction for the special case $B = \alpha A$ with $\alpha>0$ being some scaling constant in the mass lumping preconditioner $\tilde{\mathcal{B}}$. For simplicity, we choose $\alpha=1$ for the remaining part of this paper.

Notice that 
$$
\begin{pmatrix}
 \tau A&M\\
 \bar{M}&-\tau B 
\end{pmatrix}^{-1}
=
\begin{pmatrix}
 I&-(\tau A)^{-1}M\\
 0&I 
\end{pmatrix}
\begin{pmatrix}
 (\tau A)^{-1}&0\\
 0&(-\tau B-\bar{M}(\tau A)^{-1}M)^{-1} 
\end{pmatrix}
\begin{pmatrix}
 I&0\\
 -\bar{M}(\tau A)^{-1}&I 
\end{pmatrix}
$$

\begin{align*}
\tilde{\mathcal{B}}^{-1}\mathcal{A} &= 
\begin{pmatrix}
 \tau A&M\\
 \bar{M}&-\tau B 
\end{pmatrix}^{-1}
\begin{pmatrix}
 \tau A&{M}\\
 M&-\tau B 
\end{pmatrix}
=\begin{pmatrix}
 I+X&0\\
 (-\tau B-\bar{M}(\tau A)^{-1}M)^{-1}({M}-\bar M)&I 
\end{pmatrix}, \end{align*}
where
$$
X = (\tau A)^{-1}M(\tau B+\bar{M}(\tau A)^{-1}M)^{-1}({M}-\bar M).
$$
Then, any eigenvalue of $\tilde{\mathcal{B}}^{-1}\mathcal{A}$ is either 1 or $1+\lambda$ for some $\lambda \in\sigma(X)$ where $\sigma(X)$ represents the spectrum of $X$.

To derive the bounds of $\sigma(X)$, we first recall the Sherman-Morrison-Woodbury formula
$$
(A+UV^{T})^{-1} = A^{-1}-A^{-1}U(I+V^{T}A^{-1}U)^{-1}V^{T}A^{-1},
$$ 
and the following identity
\begin{equation}
V^{T}(A+UV^T)^{-1}U=(I+(V^{T}A^{-1}U)^{-1})^{-1}.
\label{smw-identity}
\end{equation}
Appying (\ref{smw-identity}), we get
\begin{align*}
X &= (\tau A)^{-1}M(\tau B+\bar{M}(\tau A)^{-1}M)^{-1}\bar{M}[\bar{M}^{-1}({M}-\bar M)]\\[5pt]
   & = \left[I+((\tau A)^{-1}M(\tau B)^{-1}\bar{M})^{-1}\right]^{-1}(\bar{M}^{-1}M-I)\\[5pt]
   & = \left(I+\tau^2(A^{-1}MB^{-1}\bar{M})^{-1}\right)^{-1}(\bar{M}^{-1}M-I)\\
   & = \left(I+\tau^2\bar{M}^{-1}BM^{-1}A\right)^{-1}(\bar{M}^{-1}M-I).
\end{align*}

\begin{lemma}[Spectrum of X] \label{eigenvaluex}
If $A = B$, and $C_1<1$ is the constant in (\ref{property1-M}), then
$$
\sigma(X)\subset(C_1-1,0].
$$
\end{lemma}
\begin{proof}
Let $(\lambda, \utilde w)$ be an eigenpair of the matrix $X$, then

\begin{eqnarray*}
X\utilde w =\lambda \utilde w &\Leftrightarrow& (\bar M^{-1}M-I)\utilde w = \lambda (I+\tau^2 \bar M^{-1}A M^{-1} A)\utilde w\\
&\Leftrightarrow& (M-\bar M)\utilde w = \lambda (\bar M+\tau^2 A M^{-1} A)\utilde w\\
&\Leftrightarrow &\lambda = \frac{((M-\bar M)\utilde w,\utilde w)}{((\bar M+\tau^2 A M^{-1} A)\utilde w,\utilde w)}\\
&\Leftrightarrow& 0\geq\lambda> (C_1-1)\frac{(\bar M\utilde w, \utilde w)}{(\bar M\utilde w,\utilde w)+(\tau^2A M^{-1}A\utilde w,\utilde w)}\\
&\Leftrightarrow& C_1-1<\lambda\leq 0.
\end{eqnarray*}
\end{proof}

By Lemma \ref{eigenvaluex} and the relation between $\tilde{\mathcal{B}}^{-1} \mathcal{A}$ and $X$, we obtain the following result.

\begin{theorem}[Spectrum of $\tilde{\mathcal{B}}^{-1} \mathcal{A}$] \label{eigentBA}
Let $\mathcal{A}$ and $\tilde{\mathcal{B}}$ be given by (\ref{2by2-block-matrix-form}) and (\ref{modified-mass-lumping-preconditioner-block}), respectively. If $A=B$, we have 
$$
\sigma( \tilde{\mathcal{B}}^{-1} \mathcal{A} ) \subset(C_1, 1].
$$
\end{theorem}

The following result can be proved by using the same proof for the Corollary 5.11 in \cite{bansch2011preconditioning} with $\epsilon_0 = C_1/2$.

\begin{corollary}[Convergence rate of GMRes for $\mathcal{A}\tilde{\mathcal{B}}^{-1}$]
GMRes's iteration converges for system $\mathcal{A}\tilde{\mathcal{B}}^{-1}$ with an asymptotic linear convergence rate bounded by
$$
\theta = \frac{1-C_1/2}{1+C_1},
$$
 Moreover, 
$$
\frac{\|r_k\|}{\|r_0\|}\leq C_0\theta^k,
$$
with $C_0=2^{2m-1}C^m_R{\rm dim}\mathbb{V}_{\mathcal{T}}/C_1^{m-1}$. 
\end{corollary}

\begin{remark}[Preconditioning for $\tau < C h^2$] \label{small-tau} By inverse inequality, we have $\tau \|u\|_A^2 \leq \tau h^{-2}\|u\|_M^2$. So if $\tau < h^2$, then $M$ will dominate $A$.
For the case when $A\neq \alpha B$ and $\tau< Ch^2$,  we rewrite the system in the following form
\begin{equation}
\begin{pmatrix}
 M & \tau A\\
 -\tau B & M
\end{pmatrix}
\begin{pmatrix}
 \utilde{u}\\
 \utilde{v}
\end{pmatrix}
=
\begin{pmatrix}
 \utilde{f}\\
 \utilde{g}
\end{pmatrix},
\end{equation}
and consider the diagonal preconditioner
\begin{equation}\label{block-diagonal-preconditioner}
\mathcal{B}_d = 
\begin{pmatrix}
 M & 0\\
 0 & M
\end{pmatrix}.
\end{equation}

To give an estimate of the spectrum bound for the preconditioned system, we let
\begin{equation*}
\begin{pmatrix}
 M & 0\\
 0 & M
\end{pmatrix}^{-1}
\begin{pmatrix}
 M & \tau A\\
 -\tau B & M
\end{pmatrix}
= I +E_d
\end{equation*}
where
\begin{equation}
E_d = 
\begin{pmatrix}
 M & 0\\
 0 & M
\end{pmatrix}^{-1}
\begin{pmatrix}
 0 & \tau A\\
 -\tau B & 0
\end{pmatrix}
\end{equation}
Let $\left(\lambda, (\utilde x^{T}, \utilde y^{T})\right)$ be an eigenpair of $E_d$, i.e.,
\begin{equation*}
\begin{pmatrix}
 0 & \tau A\\
 -\tau B & 0
\end{pmatrix}
\begin{pmatrix}
\utilde x\\
\utilde y
\end{pmatrix}
=\lambda
\begin{pmatrix}
 M & 0\\
 0 & M
\end{pmatrix}
\begin{pmatrix}
\utilde x\\
\utilde y
\end{pmatrix}
\end{equation*}
Taking the inner product of the above equation with the vector $(\utilde x^T, \utilde y^T)$ we obtain 
$$
(\tau A\utilde y,\utilde x) - (\tau B\utilde x, \utilde y) = \lambda \left[(M\utilde x,\utilde x)+(M\utilde y,\utilde y)\right].
$$
Denote $\|\cdot \|_M^2=(M\cdot, \cdot)$, we have
$$
|\lambda| = \frac{\tau\left|(A\utilde y,\utilde x)-(B\utilde x,\utilde y)\right|}{\|\utilde x\|^2_M+\|\utilde y\|^2_M}
$$
Since $M$ is symmetric positive definite,
\begin{eqnarray*}
|(A\utilde y,\utilde x)| &\leq& \left|(M^{-\frac{1}{2}}AM^{-\frac{1}{2}} M^{\frac{1}{2}}\utilde y, M^{\frac{1}{2}}\utilde x)\right|\\
&\leq& \|M^{-\frac{1}{2}}AM^{-\frac{1}{2}}\|\|\utilde y\|_M\|\utilde x\|_M\\
&\leq& \|M^{-\frac{1}{2}}AM^{-\frac{1}{2}} \|\frac{\|\utilde x\|^2_M+\|\utilde y\|^2_M}{2}.
\end{eqnarray*}
Similarly,
$$
|(B\utilde x,\utilde y)| \leq \|M^{-\frac{1}{2}}BM^{-\frac{1}{2}} \|\frac{\|\utilde x\|^2_M+\|\utilde y\|^2_M}{2}
$$
Then, by Lemma \ref{spectral4ABM},
\begin{align*}
|\lambda|&\leq \frac{\tau}{2}(\|M^{-\frac{1}{2}}AM^{-\frac{1}{2}}\|+\|M^{-\frac{1}{2}}BM^{-\frac{1}{2}}\|)\\
&\leq\frac{\tau}{2}\left(\frac{\lambda_{\max}(A)}{\lambda_{\min}(M)}+\frac{\lambda_{\max}(B)}{\lambda_{\min}(M)}\right)\\
&\leq\frac{\tau}{2}\left(\frac{C_A}{C_Mh^2}+\frac{C_B}{C_Mh^2} \right) = \frac{\tau(C_A+C_B)}{2C_Mh^2} < 1,
\end{align*}
if $\tau<Ch^2$, where $C=2C_M/(C_A+C_B)$. Hence, all the eigenvalues of the preconditioned system can be bounded in $B(1, \rho(E_d))$ with $\rho(E_d)<1$.

\end{remark}


\section{Numerical Experiments}
\label{sec:Numerical experiments}
In this section, we provide numerical experiments to demonstrate the performance of multigrid method and the preconditioned flexible GMRes method (without restart) using mass lumping preconditioners. The MATLAB adaptive finite element package $i$FEM \cite{long2009ifem} is used for all experiments. We choose the following two test problems from \cite{bansch2011preconditioning}. 

Consider the L-shaped domain $\Omega=(-1, 1)^2\backslash [0, 1)^2$, set $g=0$, $f=1$, and choose the Dirichlet boundary condition in both examples. For the diffusion coefficients $a(x_1, x_2),\ b(x_1, x_2)$, we choose
\begin{enumerate}
\item {\textbf{Nice problem}}: 
$$
a(x_1, x_2)=1,\quad b(x_1, x_2)=\left\{\begin{array} {cc} 0.6,\ & \mbox{if}\ x_2<x_1,\\ 1.2, & \mbox{otherwise}.\end{array}\right.
$$

\item {\textbf{Degenerate problem}}: 
$$
a(x_1, x_2)=0.1|x_1|+|x_2|, \quad b(x_1, x_2)=10+3\sin(5\pi x_1)\sin(8\pi x_2).
$$
\end{enumerate}

The range of parameter $\tau\in [10^{-4}, 1]$, which corresponds to the time-step size $\Delta t\in [10^{-8}, 1]$. The stopping criterion of the multigrid solver and GMRes iteration is chosen to be the relative residual error in $L^2$ norm less than $10^{-7}$. The initial guess for the iterative methods are chosen randomly. For time-dependent problems, we can use the approximation in the previous time step as the initial guess and our solvers can converge in just a few steps. All computations are done using MATLAB R2015b on a desktop computer with a 3.7 GHz Intel processor and 32 GB RAM.

\subsection{Results of multigrid solver with collective smoothers}
\label{sec:MGsolver_results}
We compare the performance of multigrid method with collective Gauss-Seidel and collective Jacobi smoother (with damping parameter $\vartheta = 0.8$) for solving (\ref{2by2-block-matrix-form}). We choose this particular Jacobi damping parameter because it gives the best performance compared with other values we have tried. We use multigrid V-cycle with one pre-smoothing and one post-smoothing (i.e., V(1,1)) which achieves the best efficiency in terms of CPU time. The number of iterations and CPU time (in parentheses) are summarized in Tables \ref{MGex1} and \ref{MGex2}. We observe that multigrid methods converge uniformly with respect to $h$. The convergence rate is close to uniform with respect to $\tau$ except for some deterioration when $\tau$ is around $10^{-4}$ or $10^{-5}$ as shown in Table \ref{MGex1}. For very small $\tau$, the block system is dominated by the two mass matrix blocks, and multigrid performs better as shown in Tables \ref{MGex1} and \ref{MGex2}. Although it is not covered by our analysis in Section \ref{sec:convergence}, numerical results in Table \ref{MGex2} show that multigrid methods with collective smoothers are robust for problems with degenerate diffusion coefficient. 

\setlength\extrarowheight{5pt}
\begin{table}
\begin{center}
\caption{Example 1: Iteration counts and CPU time (in seconds) of CGS/CJ-MG}\label{MGex1}
\begin{tabularx}{\linewidth}{m{22mm} * {10}{X}}
\hline
\multirow{2}{*}{solver} & \multirow{2}{*}{$h$} &\multicolumn{8}{c}{$\tau$} \\
                          \cline{3-10}
                                        &   &1e-0&1e-1&1e-2&1e-3&1e-4 & 1e-5 & 1e-6 &1e-7\\ 
\hline
\multirow{3}{*}{CGS-MG} &{$1/64$} &  8 (0.31) &    8 (0.31)  &    8 (0.31)  &    8 (0.31) &    13 (0.49) & 8 (0.31) & 7 (0.27) & 7 (0.27) \\
                                  &{$1/128$} &  8\;\; (1.5)  &   8\;\; (1.5)  & 8\;\; (1.5)   &   8\;\; (1.5)   &   11 (2.0) & 12 (2.2) & 7\;\; (1.3) & 7\;\; (1.3) \\
                                  &{$1/256$} & 8\;\; (6.6)  &   8\;\; (6.6)   &   8\;\; (6.6)    &   8\;\; (6.6)   &   10 (8.3) &  11 (9.1) & 9\;\; (7.5) & 7\;\; (5.9) \\
\hline
\multirow{3}{*}{CJ-MG} &{$1/64$} &15 (0.51) &15 (0.50) &15 (0.51)&15 (0.50)&14 (0.47) & 12 (0.40) & 13 (0.43) & 13 (0.43) \\
			     &{$1/128$} &15 (2.3)&15 (2.2)&15 (2.2)&15 (2.2)&14 (2.0) & 12 (1.8) & 13 (1.9) & 13 (1.9) \\
			     &{$1/256$} &15\;\; (10)&15\;\; (10)&15\;\; (10) & 15\;\; (10) & 15\;\; (10) & 14 (9.6) & 12 (8.2) & 13 (9.0)\\
\hline
\end{tabularx}
\end{center}
\end{table}%

\setlength\extrarowheight{5pt}
\begin{table}
\begin{center}
\caption{Example 2: Iteration counts and CPU time (in seconds) of CGS/CJ-MG}\label{MGex2}
\begin{tabularx}{\linewidth}{m{22mm} * {10}{X}}
\hline
\multirow{2}{*}{solver} & \multirow{2}{*}{$h$} &\multicolumn{8}{c}{$\tau$} \\
                          \cline{3-10}
                                        &   &1e-0&1e-1&1e-2&1e-3&1e-4 & 1e-5 & 1e-6 &1e-7\\ 
\hline
\multirow{3}{*}{CGS-MG} &{$1/64$} &  8 (0.31) &    8 (0.31)  &    8 (0.31)  &    8 (0.31) &    10 (0.38) & 11 (0.42) & 7 (0.27) & 7 (0.27) \\
                                  &{$1/128$} &  8\;\; (1.5)  &   8\;\; (1.5)  &    8\;\; (1.5)   &   8\;\; (1.5)   &   8\;\; (1.5) & 12\;\; (2.2) & 9\;\; (1.7) & 7\;\; (1.3) \\
                                  &{$1/256$} & 8\;\; (6.8)  &   8\;\; (6.7)   &   8\;\; (6.7)    &   8\;\; (6.7)   &   8\;\; (6.8) &  10\;\; (8.4) & 12\;\; (10) & 7\;\; (5.9) \\
\hline
\multirow{3}{*}{CJ-MG} &{$1/64$} &15 (0.53) &15 (0.52) &15 (0.51)&15 (0.54)&14 (0.48) & 12 (0.41) & 14 (0.48) & 13 (0.44) \\
			     &{$1/128$} &15 (2.2) & 15 (2.2) & 15 (2.2) & 15 (2.2) & 14 (2.1) & 13 (1.9) & 13 (1.9) & 14 (2.1) \\
			     &{$1/256$} &15\;\; (11) & 15\;\; (10) & 15\;\; (10) & 15\;\; (10) & 15\;\; (10) & 14\;\; (10) & 12\;\; (8.3) & 14 (8.6)\\
\hline
\end{tabularx}
\end{center}
\end{table}%

\subsection{Results of GMRes with preconditioners solved inexactly by CGS-MG}\label{section-results-cgs-mg}
\label{subsec:6.2}
In this subsection, we compare the performance of the preconditioned GMRes method using three different preconditioners $V_{\mathcal{B}}(1,1), V_{\tilde{\mathcal{B}}}(1,1), V_{\mathcal{A}}(1,1)$. Here, we use $V_{\mathcal{B}}(1,1),\ V_{\tilde{\mathcal{B}}}(1,1), V_{\mathcal{A}}(1,1)$ to represent the inexact preconditioners corresponding to $\mathcal{B},\tilde{\mathcal{B}}$ and $\mathcal{A}$ respectively. For example, $V_{\mathcal{B}}(1,1)$ is one CGS-MG V-cycle with one pre-smoothing and one post-smoothing. As we can see from Tables \ref{GMRes-MG-Vanka-ex1} and \ref{GMRes-MG-Vanka-ex2} that GMRes method converges uniformly with respect to both $h$ and $\tau$. The CPU time for GMRes preconditioned by $V_{\mathcal{B}}(1,1)$ is approximately half of the time when preconditioned by $V_{\mathcal{A}}(1,1)$ which clearly demonstrate the efficiency of mass lumping.

\begin{table}
\caption{Example 1: Iteration counts and CPU time (in seconds) of GMRes with one V(1,1) CGS-MG} 
\begin{center}
\begin{tabularx}{\linewidth}{m{25mm} *7{X}}
\hline
\multirow{2}{*}{preconditioner} & \multirow{2}{*}{$h$} &\multicolumn{5}{c}{$\tau$} \\
                          \cline{3-7}
                                        &   &1e-0&1e-1&1e-2&1e-3&1e-4\\ 
\hline
\multirow{3}{*}{$V_{\mathcal{B}}(1,1)$} & {$1/64$}  &  8 (0.16) &    8 (0.16)  &   8 (0.16)  &   7 (0.14) &   12 (0.25)\\
         					&{$1/128$} &  8 (0.76)  &   8 (0.72)  &    8 (0.72)   &   7 (0.63)   &   8 (0.72)\\
             					& {$1/256$}  & 8 (3.5)  &   8 (3.4)   &   8 (3.4)    &   7 (3.0)   &   7 (3.0)\\
\hline
\multirow{3}{*}{$V_{\tilde{\mathcal{B}}}(1,1)$} & {$1/64$}& 8 (0.24) &    8 (0.21)  &   8 (0.22)  &   7 (0.19) &   9 (0.24)\\						        &{$1/128$} & 8 (1.0)  &   8 (1.0)  &  8 (0.98)   &   7 (0.87)   &   7 (0.86)\\
						&{$1/256$} & 8 (4.4)  &   8 (4.4)   &   8 (4.4)    &   7 (3.9)   &   7 (3.9)\\
\hline						
\multirow{3}{*}{$V_{\mathcal{A}}(1,1)$} & {$1/64$}& 8 (0.33) & 8 (0.33) &8 (0.34) &8 (0.33) &7 (0.29)\\									&{$1/128$} &  8 (1.5) &    8 (1.5)  &    8 (1.5)   &   8 (1.5)   &   8 (1.5)\\
						&{$1/256$} & 8 (6.9)  & 8 (6.9) &   8 (6.9)&  7 (6.1)	&  8 (6.9)\\	
\hline
\end{tabularx}
\end{center}
\label{GMRes-MG-Vanka-ex1}
\end{table}%

\begin{table}
\caption{Example 2: Iteration counts and CPU time (in seconds) of GMRes with one V(1,1) CGS-MG}
\begin{center}
\begin{tabularx}{\linewidth}{m{25mm} *7{X}}
\hline
 \multirow{2}{*}{preconditioner} & \multirow{2}{*}{$h$} &\multicolumn{5}{c}{$\tau$} \\
                          \cline{3-7}
                                        &   &1e-0&1e-1&1e-2&1e-3&1e-4\\ 
\hline
\multirow{3}{*}{$V_{\mathcal{B}}(1,1)$} & {$1/64$}  & 8 (0.17) &    8 (0.17)  &   8 (0.17)  &   8 (0.17) &   14 (0.30)\\
         					&{$1/128$} &  8 (0.77)  &   8 (0.73)  &    8 (0.73)   &   8 (0.73)   &   9 (0.82)\\
             					& {$1/256$}  & 8 (3.5)  &   8 (3.4)   &   8 (3.4)    &   8 (3.4)   &   7 (3.0)\\
\hline
\multirow{3}{*}{$V_{\tilde{\mathcal{B}}}(1,1)$} & {$1/64$}&8 (0.22) & 8 (0.21)  &   8 (0.21)  &   8 (0.22) &   8 (0.24)\\				       					&{$1/128$} & 8 (1.0)  &   8 (0.98)  &    8 (0.98)   &   8 (0.98)   &   7 (0.86)\\
						       &{$1/256$} &  8 (4.5) & 8 (4.5)   &   8 (4.5)    &   8 (4.5)   &   7 (3.9)\\
\hline
\multirow{3}{*}{$V_{\mathcal{A}}(1,1)$} & {$1/64$}& 8 (0.32) &   8 (0.33) &   8 (0.34) &  8 (0.33) &  9 (0.37)\\								         &{$1/128$} &  8 (1.5) &8 (1.5) &8 (1.5) &8 (1.5) &8 (1.5)\\
						&{$1/256$} & 8 (6.9) &8 (6.9) &8 (6.9) &8 (6.9) &8 (6.9)\\
\hline
\end{tabularx}
\end{center}
\label{GMRes-MG-Vanka-ex2}
\end{table}

\subsection{Results of GMRes with preconditioners solved inexactly by DGS-MG}\label{section-results-dgs-mg}
In this subsection, we present the numerical results for GMRes method using inexact mass lumping preconditioners $V_{\mathcal{B}}(1, 1)$ and $V_{\tilde{\mathcal{B}}}(1, 1)$. Namely, the preconditioner systems $\mathcal{B}$ and $\tilde{\mathcal{B}}$ are solved approximately by one multigrid V-cycle using the decoupled smoother proposed in Section \ref{subsection:mg_dgs} with one pre-smoothing and one post-smoothing. Here, we choose the Jacobi damping parameter to be $0.5$. Our numerical experiments show that the preconditioners have similar performance when the Jacobi damping parameter varies between $0.4$ and $0.6$. As we can see from Tables \ref{GMRes-MG-DGS-ex1} and \ref{GMRes-MG-DGS-ex2} that GMRes converges uniformly when varying $h$ and $\tau$.

 \begin{table}
\caption{Example 1: Iteration counts and CPU time (in seconds) of GMRes with one V(1, 1) DGS-MG}
\begin{center}
\begin{tabularx}{\linewidth}{m{25mm} *7{X}}
\hline
 \multirow{2}{*}{preconditioner} & \multirow{2}{*}{$h$} &\multicolumn{5}{c}{$\tau$} \\
                          \cline{3-7}
                                        &   &1e-0&1e-1&1e-2&1e-3&1e-4\\ 
\hline
\multirow{3}{*}{$V_{\mathcal{B}}(1, 1)$} & {$1/64$}  & 22 (0.28)& 22 (0.28) & 21 (0.26)  & 19 (0.25) & 20 (0.25)\\
         					&{$1/128$} &  22 (1.1) & 22 (1.1) & 21 (1.1) & 20 (1.0) & 19 (0.98)\\
             					& {$1/256$}  & 22 (5.4) & 22 (5.3) & 22 (5.3) & 21 (5.1) & 19 (4.6)\\
\hline
\multirow{3}{*}{$V_{\tilde{\mathcal{B}}}(1, 1)$} & {$1/64$}&22 (0.29) &  22 (0.30)  & 21 (0.28)  & 19 (0.25) & 17 (0.40)\\						                        &{$1/128$} & 22 (1.2) &   22 (1.2)  & 21 (1.1)  & 20 (1.1) & 19 (1.0)\\
						       & {$1/256$} & 22 (5.5) &  22 (5.5)  &  22 (5.5)  &  21 (5.2) & 19 (4.8)\\
\hline
\end{tabularx}
\end{center}
\label{GMRes-MG-DGS-ex1}
\end{table}

\begin{table}
\caption{Example 2: Iteration counts and CPU time (in seconds) of GMRes with one V(1, 1) DGS-MG}
\begin{center}
\begin{tabularx}{\linewidth}{m{25mm} *7{X}}
\hline
\multirow{2}{*}{preconditioner} & \multirow{2}{*}{$h$} &\multicolumn{5}{c}{$\tau$} \\
                          \cline{3-7}
                                        &   &1e-0&1e-1&1e-2&1e-3&1e-4\\ 
\hline
\multirow{3}{*}{$V_{\mathcal{B}}(1, 1)$} & {$1/64$}  & 23 (0.29) &   22 (0.28)  &   21 (0.27)  &  19 (0.24) &  20 (0.25)\\
         					&{$1/128$} &  21 (1.1)  &  22 (1.1)  &  22 (1.1)   &  20 (1.0)   &   18 (0.92)\\
             					& {$1/256$}  & 21 (5.0)  &  21 (5.0)   &  21 (5.0)    &  20 (4.8)   & 18 (4.3)\\
\hline
\multirow{3}{*}{$V_{\tilde{\mathcal{B}}}(1, 1)$} & {$1/64$}& 23 (0.30)& 23 (0.30)& 22 (0.29)& 19 (0.26)& 17 (0.23)\\										&{$1/128$} & 22 (1.2)& 23 (1.2)& 22 (1.2)& 20 (1.1)& 17 (0.94)\\
						       & {$1/256$} &21 (5.3) & 22 (5.5)& 21 (5.3)& 20 (5.0)& 18 (4.5)\\
\hline
\end{tabularx}
\end{center}
\label{GMRes-MG-DGS-ex2}
\end{table}

\subsection{Choice of multigrid parameters for solving preconditioner systems}
In the following, we compare the performance of the preconditioner $\mathcal{B}$ solved by CGS-MG or DGS-MG with different types of cycles and different number of smoothing steps. From Tables \ref{revised_1} - \ref{revised_4}, we can see that increasing the number of smoothing steps can reduce the number of GMRes iterations. However, the CPU time is not always decreasing since more smoothing steps may also increase the computational cost. The GMRes iteration numbers are similar when using W cycle or V cycle. Since W cycle is usually more expensive, its efficiency in terms of CPU time is not as good as V cycle in our experiment. For practical purposes, we suggest using V cycle with one pre- and one post-smoothing step for solving the preconditioner systems discussed in this work. 

\setlength\extrarowheight{5pt}
\begin{table}
\begin{center}
\footnotesize\setlength{\tabcolsep}{2.5pt}
\caption{Example 1: Iteration count, CPU time (in seconds), and convergence factor (italics) of GMRes preconditioned by one $V_
\mathcal{B}(k,k)$ CGS-MG, $k=1,2,3$}\label{revised_1}
\begin{tabularx}{\linewidth}{m{20mm} *7{X}}
\hline
\multirow{2}{*}{precond} & \multirow{2}{*}{$h$} &\multicolumn{5}{c}{$\tau$} \\
                          \cline{3-7}
                                        &   &1e-0&1e-1&1e-2&1e-3&1e-4\\ 
\hline
\multirow{3}{*}{$V_\mathcal{B}(1,1)$} &\multirow{2}{*}{$1/64$} & 8 (0.16) &   8 (0.16)  &    8 (0.16)  &  7 (0.14) &  12 (0.25) \\
& & {\itshape 0.09} & {\itshape 0.09} & {\itshape 0.09} & {\itshape 0.10} & {\itshape 0.23}\\
\cline{2-7}
                                  &\multirow{2}{*}{$1/128$} &  8 (0.76)  &  8 (0.72)  &  8 (0.72)   & 7 (0.63)   &  8 (0.72) \\
& & {\itshape 0.09} & {\itshape 0.09} & {\itshape 0.09} & {\itshape 0.10} & {\itshape 0.12}\\                               
\cline{2-7}
                                  &\multirow{2}{*}{$1/256$} & 8 (3.5)  &  8 (3.4)   &  8 (3.4)    &  7 (3.0)   &  7 (3.0)\\
& & {\itshape 0.09} & {\itshape 0.09} & {\itshape 0.09} & {\itshape 0.09} & {\itshape 0.10}\\                                 
\hline
\multirow{3}{*}{$V_\mathcal{B}(2,2)$} &\multirow{2}{*}{$1/64$} & 6 (0.14) &  6 (0.14)  &   5 (0.12)  & 6 (0.14) &  11 (0.26) \\
& & {\itshape 0.09} & {\itshape 0.09} & {\itshape 0.09} & {\itshape 0.10} &{\itshape 0.23}\\
\cline{2-7}
                                  &\multirow{2}{*}{$1/128$} &  6 (0.67)  &   6 (0.63)  &  5 (0.53)   & 5 (0.53)   &   6 (0.63) \\
& & {\itshape 0.09} & {\itshape 0.09} & {\itshape 0.09} & {\itshape 0.10} & {\itshape 0.12}\\                               
\cline{2-7}
                                  &\multirow{2}{*}{$1/256$} & 6 (3.1)  &  6 (3.0)   &   6 (3.0)    &   5 (2.5)   &  5 (2.5)\\
& & {\itshape 0.09} & {\itshape 0.09} & {\itshape 0.09} & {\itshape 0.09} & {\itshape 0.10}\\     
\hline
\multirow{3}{*}{$V_\mathcal{B}(3,3)$} &\multirow{2}{*}{$1/64$} & 6 (0.16) &   5 (0.14)  &   5 (0.14)  &  5 (0.14) &  11 (0.30) \\
& & {\itshape 0.02} & {\itshape 0.02} & {\itshape 0.04} & {\itshape 0.4} & {\itshape 0.24}\\
\cline{2-7}
                                  &\multirow{2}{*}{$1/128$} &  6 (0.75)  &   5 (0.59)  &   5 (0.59)   & 4 (0.48)   &   6 (0.70) \\
& & {\itshape 0.02} & {\itshape 0.02} & {\itshape 0.02} & {\itshape 0.03} & {\itshape 0.09}\\                               
\cline{2-7}
                                  &\multirow{2}{*}{$1/256$} & 6 (3.5)  &  5 (2.9)   &   5 (2.9)    &   4 (2.3)   &   5 (2.9)\\
& & {\itshape 0.02} & {\itshape 0.02} & {\itshape 0.02} & {\itshape 0.02} & {\itshape 0.05}\\    
\hline
\end{tabularx}
\end{center}
\end{table}

\setlength\extrarowheight{5pt}
\begin{table}
\begin{center}
\footnotesize\setlength{\tabcolsep}{2.5pt}
\caption{Example 1: Iteration count, CPU time (in seconds), and convergence factor (italics) of GMRes preconditioned by one $W_\mathcal{B}(k,k)$ CGS-MG, $k=1,2,3$}\label{revised_2}
\begin{tabularx}{\linewidth}{m{20mm} *7{X}}
\hline
\multirow{2}{*}{precond} & \multirow{2}{*}{$h$} &\multicolumn{5}{c}{$\tau$} \\
                          \cline{3-7}
                                        &   &1e-0&1e-1&1e-2&1e-3&1e-4\\ 
\hline
\multirow{3}{*}{$W_\mathcal{B}(1,1)$} &\multirow{2}{*}{$1/64$} & 8 (0.32) &   8 (0.32)  &   7 (0.28)  &  7 (0.28) &  12 (0.47) \\
& & {\itshape 0.07} & {\itshape 0.07} & {\itshape 0.07} & {\itshape 0.09} & {\itshape 0.23}\\
\cline{2-7}
                                  &\multirow{2}{*}{$1/128$} &  8 (1.2)  &  8 (1.1)  &  7 (1.0)   & 7 (1.0)   &  8 (1.1) \\
& & {\itshape 0.07} & {\itshape 0.07} & {\itshape 0.07} & {\itshape 0.07} & {\itshape 0.12}\\                               
\cline{2-7}
                                  &\multirow{2}{*}{$1/256$} & 8 (5.1)  &  8 (4.9)   &   8 (4.9)    &   7 (4.3)   &   7 (4.3)\\
& & {\itshape 0.07} & {\itshape 0.07} & {\itshape 0.07} & {\itshape 0.07} & {\itshape 0.08}\\                                 
\hline
\multirow{3}{*}{$W_\mathcal{B}(2,2)$} &\multirow{2}{*}{$1/64$} & 6 (0.28) &   6 (0.28)  &   5 (0.23)  &  5 (0.24) &  11 (0.51) \\
& & {\itshape 0.03} & {\itshape 0.03} & {\itshape 0.03} & {\itshape 0.05} & {\itshape 0.24}\\
\cline{2-7}
                                  &\multirow{2}{*}{$1/128$} &  6 (1.0)  &  6 (1.0)  &  5 (0.86)   & 5 (0.87)   &  6 (1.0) \\
& & {\itshape 0.03} & {\itshape 0.03} & {\itshape 0.03} & {\itshape 0.03} & {\itshape 0.08}\\                               
\cline{2-7}
                                  &\multirow{2}{*}{$1/256$} & 6 (4.5)  &  6 (4.4)   &   5 (3.7)    &   5 (3.7)   &   5 (3.7)\\
& & {\itshape 0.03} & {\itshape 0.03} & {\itshape 0.03} & {\itshape 0.03} & {\itshape 0.04}\\     
\hline
\multirow{3}{*}{$W_\mathcal{B}(3,3)$} &\multirow{2}{*}{$1/64$} & 5 (0.27) &   5 (0.27)  &   4 (0.22)  &  5 (0.27) &  11 (0.59) \\
& & {\itshape 0.02} & {\itshape 0.02} & {\itshape 0.02} & {\itshape 0.4} & {\itshape 0.24}\\
\cline{2-7}
                                  &\multirow{2}{*}{$1/128$} &  5 (1.0)  &   5 (0.98)  &   5 (0.98)   & 4 (0.79)   &   6 (1.2) \\
& & {\itshape 0.02} & {\itshape 0.02} & {\itshape 0.02} & {\itshape 0.02} & {\itshape 0.08}\\                               
\cline{2-7}
                                  &\multirow{2}{*}{$1/256$} & 5 (4.2)  &  5 (4.2)   &  5 (4.2)    &  4 (3.4)   &  5 (4.2)\\
& & {\itshape 0.02} & {\itshape 0.02} & {\itshape 0.02} & {\itshape 0.02} & {\itshape 0.03}\\    
\hline
\end{tabularx}
\end{center}
\end{table}

\setlength\extrarowheight{5pt}
\begin{table}
\begin{center}
\footnotesize\setlength{\tabcolsep}{2.5pt}
\caption{Example 1: Iteration count, CPU time (in seconds), and convergence factor (italics) of GMRes preconditioned by one $V_\mathcal{B}(k,k)$ DGS-MG, $k=1,2,3$}
\label{revised_3}
\begin{tabularx}{\linewidth}{m{20mm} *7{X}}
\hline
\multirow{2}{*}{precond} & \multirow{2}{*}{$h$} &\multicolumn{5}{c}{$\tau$} \\
                          \cline{3-7}
                                        &   &1e-0&1e-1&1e-2&1e-3&1e-4\\ 
\hline
\multirow{3}{*}{$V_\mathcal{B}(1,1)$} &\multirow{2}{*}{$1/64$} & 22 (0.28) &   22 (0.28)  &   21 (0.26)  &  19 (0.25) &  20 (0.25) \\
& & {\itshape 0.48} & {\itshape 0.45} & {\itshape 0.44} & {\itshape 0.44} & {\itshape 0.42}\\
\cline{2-7}
                                  &\multirow{2}{*}{$1/128$} &  22 (1.1)  &  22 (1.1)  &  21 (1.1)   & 20 (1.0)   &  19 (0.98) \\
& & {\itshape 0.46} & {\itshape 0.48} & {\itshape 0.44} & {\itshape 0.43} & {\itshape 0.41}\\                               
\cline{2-7}
                                  &\multirow{2}{*}{$1/256$} & 22 (5.4)  &  22 (5.3)   &  22 (5.3)    &   21 (5.1)   &   19 (4.6)\\
& & {\itshape 0.44} & {\itshape 0.44} & {\itshape 0.45} & {\itshape 0.44} & {\itshape 0.43}\\                                 
\hline
\multirow{3}{*}{$V_\mathcal{B}(2,2)$} &\multirow{2}{*}{$1/64$} & 16 (0.25) &  15 (0.23)  &  15 (0.23)  &  14 (0.23) &  16 (0.25) \\
& & {\itshape 0.35} & {\itshape 0.32} & {\itshape 0.31} & {\itshape 0.29} & {\itshape 0.33}\\
\cline{2-7}
                                  &\multirow{2}{*}{$1/128$} &  15 (0.97)  &  15 (0.97)  &  15 (0.96)   & 14 (0.91)   &  13 (0.85) \\
& & {\itshape 0.3} & {\itshape 0.31} & {\itshape 0.33} & {\itshape 0.29} & {\itshape 0.28}\\                               
\cline{2-7}
                                  &\multirow{2}{*}{$1/256$} & 15 (4.4)  &  15 (4.4)   &  15 (4.3)    &  14 (4.1)   &  13 (3.8)\\
& & {\itshape 0.3} & {\itshape 0.3} & {\itshape 0.33} & {\itshape 0.3} & {\itshape 0.29}\\     
\hline
\multirow{3}{*}{$V_\mathcal{B}(3,3)$} &\multirow{2}{*}{$1/64$} & 13 (0.23) &  13 (0.24)  &  13 (0.24)  &  12 (0.23) &  14 (0.25) \\
& & {\itshape 0.26} & {\itshape 0.23} & {\itshape 0.26} & {\itshape 0.24} & {\itshape 0.29}\\
\cline{2-7}
                                  &\multirow{2}{*}{$1/128$} &  13 (0.98)  &  13 (0.97)  &  13 (0.97)   &  12 (0.91)   &   11 (0.84) \\
& & {\itshape 0.24} & {\itshape 0.26} & {\itshape 0.28} & {\itshape 0.24} & {\itshape 0.21}\\                               
\cline{2-7}
                                  &\multirow{2}{*}{$1/256$} & 13 (4.3)  &  13 (4.3)   &  13 (4.3)    &  12 (4.0)   &  11 (3.7)\\
& & {\itshape 0.24} & {\itshape 0.24} & {\itshape 0.25} & {\itshape 0.25} & {\itshape 0.24}\\    
\hline
\end{tabularx}
\end{center}
\end{table}

\setlength\extrarowheight{5pt}
\begin{table}
\begin{center}
\footnotesize\setlength{\tabcolsep}{2.5pt}
\caption{Example 1: Iteration count, CPU time (in seconds), and convergence factor (italics) of GMRes preconditioned by one $W_\mathcal{B}(k,k)$ DGS-MG, $k=1,2,3$}
\label{revised_4}
\begin{tabularx}{\linewidth}{m{20mm} *7{X}}
\hline
\multirow{2}{*}{precond} & \multirow{2}{*}{$h$} &\multicolumn{5}{c}{$\tau$} \\
                          \cline{3-7}
                                        &   &1e-0&1e-1&1e-2&1e-3&1e-4\\ 
\hline
\multirow{3}{*}{$W_\mathcal{B}(1,1)$} &\multirow{2}{*}{$1/64$} & 21 (0.56) &  21 (0.56)  &  20 (0.53)  &  19 (0.51) &  19 (0.51) \\
& & {\itshape 0.41} & {\itshape 0.40} & {\itshape 0.42} & {\itshape 0.41} & {\itshape 0.42}\\
\cline{2-7}
                                  &\multirow{2}{*}{$1/128$} &  21 (1.8)  &  21 (1.9)  & 21 (1.8)   & 19 (1.7)   & 19 (1.7) \\
& & {\itshape 0.40} & {\itshape 0.40} & {\itshape 0.41} & {\itshape 0.41} & {\itshape 0.41}\\                               
\cline{2-7}
                                  &\multirow{2}{*}{$1/256$} & 21 (7.4)  &  21 (7.4)   &  21 (7.4)    & 20 (7.1)   &  19 (6.7)\\
& & {\itshape 0.40} & {\itshape 0.41} & {\itshape 0.41} & {\itshape 0.42} & {\itshape 0.42}\\                                 
\hline
\multirow{3}{*}{$W_\mathcal{B}(2,2)$} &\multirow{2}{*}{$1/64$} & 15 (0.49) & 15 (0.50)  &  14 (0.46)  & 13 (0.43) &  15 (0.50) \\
& & {\itshape 0.26} & {\itshape 0.26} & {\itshape 0.27} & {\itshape 0.27} & {\itshape 0.33}\\
\cline{2-7}
                                  &\multirow{2}{*}{$1/128$} &  15 (1.7)  &  15 (1.6)  &  14 (1.5)   & 13(1.4)   &  13 (1.4) \\
& & {\itshape 0.26} & {\itshape 0.26} & {\itshape 0.27} & {\itshape 0.26} & {\itshape 0.28}\\                               
\cline{2-7}
                                  &\multirow{2}{*}{$1/256$} & 15 (6.5)  & 15 (6.5)   & 14 (6.0)    &  14 (6.1)   &  13 (5.7)\\
& & {\itshape 0.26} & {\itshape 0.26} & {\itshape 0.27} & {\itshape 0.27} & {\itshape 0.26}\\     
\hline
\multirow{3}{*}{$W_\mathcal{B}(3,3)$} &\multirow{2}{*}{$1/64$} & 12 (0.46) & 12 (0.47)  & 12 (0.46)  &  11 (0.43) &  14 (0.53) \\
& & {\itshape 0.19} & {\itshape 0.19} & {\itshape 0.19} & {\itshape 0.20} & {\itshape 0.29}\\
\cline{2-7}
                                  &\multirow{2}{*}{$1/128$} &  12 (1.6)  & 12 (1.6)  & 12 (1.6)   & 11 (1.4)   &  11 (1.4) \\
& & {\itshape 0.19} & {\itshape 0.19} & {\itshape 0.19} & {\itshape 0.20} & {\itshape 0.21}\\                               
\cline{2-7}
                                  &\multirow{2}{*}{$1/256$} & 12 (6.0)  &  12 (6.0)   &  12 (6.0)    &  11 (5.6)   & 11 (5.6)\\
& & {\itshape 0.19} & {\itshape 0.19} & {\itshape 0.19} & {\itshape 0.19} & {\itshape 0.19}\\    
\hline
\end{tabularx}
\end{center}
\end{table}%

\subsection{Results of GMRes with preconditioners $\tilde{\mathcal{B}}$ and $\mathcal{B}_d$}
The theoretical analysis in Remark \ref{small-tau} shows that when $\tau < Ch^2$, the block diagonal preconditioner $\mathcal{B}_d$ defined by (\ref{block-diagonal-preconditioner}) is also good for GMRes method. In Tables \ref{GMRes-different-tau-ex1} and \ref{GMRes-different-tau-ex2}, we demonstrate the performance of inexact $\mathcal{B}_d$ preconditioner when $\tau$ is very small. We use $GS_{\mathcal{B}_d}(3)$ to represent three steps of Gauss-Seidel relaxation applied to $\mathcal{B}_d$. It can be observed from the numerical results that $GS_{\mathcal{B}_d}(3)$ is very efficient when $\tau$ is very small. However, we would like to point out that from the approximation point of view, $\tau$ should be of order $h^2$. Hence, in practice, a reasonable choice of $\tau$ may not be too small.

 \begin{table}
\caption{Example 1: Iteration counts and CPU time (in seconds) of GMRes with preconditioner $\tilde{\mathcal{B}}$ or ${\mathcal{B}_d}$ ( $*$ means no convergence within $200$ iterations) }
\begin{center}
\begin{tabularx}{\linewidth}{m{25mm} *6{X}}
\hline
\multirow{2}{*}{preconditioner} & \multirow{2}{*}{$h$} &\multicolumn{4}{c}{$\tau$} \\
                          \cline{3-6}
                                        &   &1e-4&1e-5&1e-6&1e-7\\ 
\hline
\multirow{3}{*}{$V_{\tilde{\mathcal{B}}}(1,1)$} 
& $1/64$ &   9 (0.24) &   12 (0.31) & 15 (0.39) & 15 (0.39) \\
&$1/128$ &   7 (0.87)   &   9 (1.1) & 14 (1.7) & 15 (1.8) \\
&$1/256$  &  7 (3.8)   &   8 (4.4) &11 (6.0)  & 15 (8.2) \\
\hline
\multirow{3}{*}{$GS_{\mathcal{B}_d}(3)$} 
& $1/64$ &   164 (0.99) &   19 (0.1) & 6 (0.03) & 4 (0.022) \\
&$1/128$ &   *   &   68 (1.3) & 10 (0.17) & 5 (0.096)  \\
&$1/256$  &  *   &   * & 29 (2.7) & 7 (0.62) \\
\hline
\end{tabularx}
\end{center}
\label{GMRes-different-tau-ex1}
\end{table}

\begin{table}
\caption{Example 2: Iteration counts and CPU time (in seconds) of GMRes with preconditioner $\tilde{\mathcal{B}}$ or ${\mathcal{B}_d}$ ( $*$ means no convergence within $200$ iterations) }
\begin{center}
\begin{tabularx}{\linewidth}{m{25mm} *6{X}}
\hline
\multirow{2}{*}{preconditioner} & \multirow{2}{*}{$h$} &\multicolumn{4}{c}{$\tau$} \\
                          \cline{3-6}
                                        &   &1e-4&1e-5&1e-6&1e-7\\ 
\hline
\multirow{3}{*}{$V_{\tilde{\mathcal{B}}}(1,1)$} 
& $1/64$ &   8 (0.21) &   13 (0.35) & 15 (0.39) & 15 (0.39) \\
&$1/128$ &   7 (0.88)   &   12 (1.5) & 15 (1.8) & 15 (1.8) \\
&$1/256$  &  7 (3.9)   &   8 (4.4) & 14 (7.7)  & 15 (8.2) \\
\hline
\multirow{3}{*}{$GS_{\mathcal{B}_d}(3)$} 
& $1/64$ &   * &   58 (0.35) & 9 (0.043) & 5 (0.026) \\
&$1/128$ &   *   &   * & 25 (0.43) & 7 (0.12)  \\
&$1/256$  &  *   &   * & 96 (10) & 12 (1.0) \\
\hline
\end{tabularx}
\end{center}
\label{GMRes-different-tau-ex2}
\end{table}%

\subsection{Robustness with respect to other boundary conditions}

In previous numerical experiments, homogeneous Dirichlet boundary conditions are prescribed. In Tables \ref{GMRes-MG-Vanka-ex1-mixed} and \ref{GMRes-MG-DGS-ex1-mixed}, we show that the proposed preconditioners also work well for problems with other boundary conditions. In particular, we consider example 1 defined at the beginning of Section \ref{sec:Numerical experiments} and choose mixed homogeneous Neumann and Dirichlet boundary conditions \ref{bd1} and \ref{bd2} where $\Gamma_N:=\{(x,y) | x=0, y\in(0,1)\}\cup \{(x,y) | x\in(0,1), y=0\}$ and $\Gamma_D:=\partial\Omega\backslash \Gamma_N$. There is a slight increase of GMRes iteration numbers when $\tau$ is large and the preconditioned systems are solved by CGS-MG.

\begin{table}
\caption{Example 1 with mixed boundary conditions: Iteration counts and CPU time (in seconds) of GMRes with one V(1,1) CGS-MG} 
\label{GMRes-MG-Vanka-ex1-mixed}
\begin{center}
\begin{tabularx}{\linewidth}{m{25mm} *7{X}}
\hline
\multirow{2}{*}{preconditioner} & \multirow{2}{*}{$h$} &\multicolumn{5}{c}{$\tau$} \\
                          \cline{3-7}
                                        &   &1e-0&1e-1&1e-2&1e-3&1e-4\\ 
\hline
\multirow{3}{*}{$V_{\mathcal{B}}(1,1)$} & {$1/64$}  &  10 (0.21) &    9 (0.20)  &   8 (0.17)  &   7 (0.15) &   12 (0.26)\\
         					&{$1/128$} &  10 (0.93)  &   10 (0.91)  &    9 (0.83)   &   7 (0.64)   &   8 (0.73)\\
             					& {$1/256$}  & 10 (4.4)  &   10 (4.3)   &   9 (3.9)    &   8 (3.4)   &   7 (3.0)\\
\hline
\multirow{3}{*}{$V_{\tilde{\mathcal{B}}}(1,1)$} & {$1/64$}& 10 (0.26) &   10 (0.26)  &   8 (0.21)  &   7 (0.20) &   9 (0.24)\\						        &{$1/128$} & 10 (1.3)  &   10 (1.3)  &  9 (1.4)   &   7 (0.87)   &   7 (0.88)\\
						&{$1/256$} & 10 (5.7)  &  10 (5.7)   &  9 (5.1)    &  8 (4.6)   &   7 (4.0)\\
\hline						
\multirow{3}{*}{$V_{\mathcal{A}}(1,1)$} & {$1/64$}& 10 (0.40) & 10 (0.40) & 8 (0.33) &8 (0.32) &7 (0.28)\\									&{$1/128$} &  10 (2.0) &    10 (2.0)  &  9 (1.7)   &   8 (1.6)   &   8 (1.6)\\
						&{$1/256$} & 10 (8.8)  & 10 (8.8) &   9 (7.9)&  8 (7.1)	&  8 (7.0)\\	
\hline
\end{tabularx}
\end{center}
\end{table}%

 \begin{table}
\caption{Example 1 with mixed boundary conditions: Iteration counts and CPU time (in seconds) of GMRes with one W(2,2) DGS-MG}
\label{GMRes-MG-DGS-ex1-mixed}
\begin{center}
\begin{tabularx}{\linewidth}{m{25mm} *7{X}}
\hline
 \multirow{2}{*}{preconditioner} & \multirow{2}{*}{$h$} &\multicolumn{5}{c}{$\tau$} \\
                          \cline{3-7}
                                        &   &1e-0&1e-1&1e-2&1e-3&1e-4\\ 
\hline
\multirow{3}{*}{$W_{\mathcal{B}}(2,2)$} & {$1/64$}  & 15 (0.53)& 15 (0.51)& 15 (0.50)  & 13 (0.44) & 15 (0.50)\\
         					&{$1/128$} &  15 (1.7) & 15 (1.7) & 15 (1.7) & 14 (1.6) & 13 (1.5)\\
             					& {$1/256$}  & 15 (6.6) & 15 (6.6) & 15 (6.6) & 14 (6.2) & 13 (5.8)\\
\hline
\multirow{3}{*}{$W_{\tilde{\mathcal{B}}}(2,2)$} & {$1/64$}&15 (0.58) &    15 (0.56)  & 15 (0.55)  & 13 (0.48) & 12 (0.45)\\						                        &{$1/128$} & 15 (1.9) &    15 (1.9)  &   15 (1.9)  &   14 (1.7) & 13 (1.6)\\
						       & {$1/256$} &15 (7.2) &  15 (7.2)  &   15 (7.2)  &   14 (6.7) &13 (6.3)\\
\hline
\end{tabularx}
\end{center}
\end{table}

\subsection{Spectrum distribution of the preconditioned systems}
In Figures \ref{spectrum4E} and \ref{spectrum4tBA}, we show the spectrum distribution of the preconditioned systems, $\mathcal{B}^{-1}\mathcal{A}$ and $\tilde{\mathcal{B}}^{-1}\mathcal{A}$, respectively, when varying $h$ (left) or $\tau$ (right) (with diffusion coefficients $a(x_1,x_2)=b(x_1,x_2)=1$). We can see that the eigenvalues are more concentrated when $h$ decreases with fixed $\tau$. They become more scattered when $\tau$ gets smaller while $h$ is fixed, but away from zero which means the original system is well conditioned. Moreover, the preconditioned system $\tilde{\mathcal{B}}^{-1}\mathcal{A}$ has more favorable spectral property than $\mathcal{B}^{-1}\mathcal{A}$ for GMRes method which is consistent with the numerical results reported in Sections \ref{section-results-cgs-mg} and \ref{section-results-dgs-mg}.

\begin{figure}
\centering
\includegraphics[width=7.85cm]{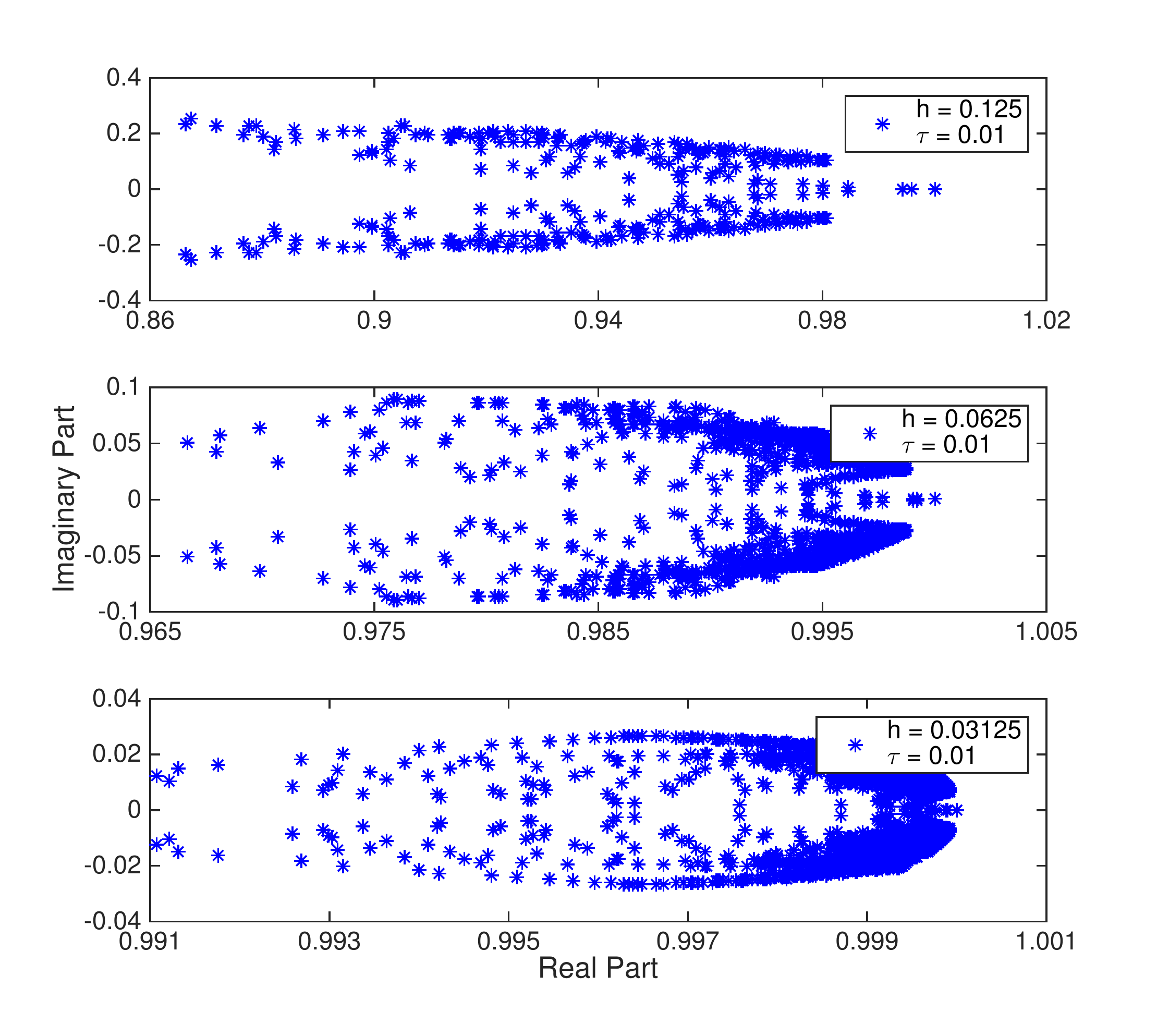}
\hspace{-0.32in}
\includegraphics[width=7.85cm]{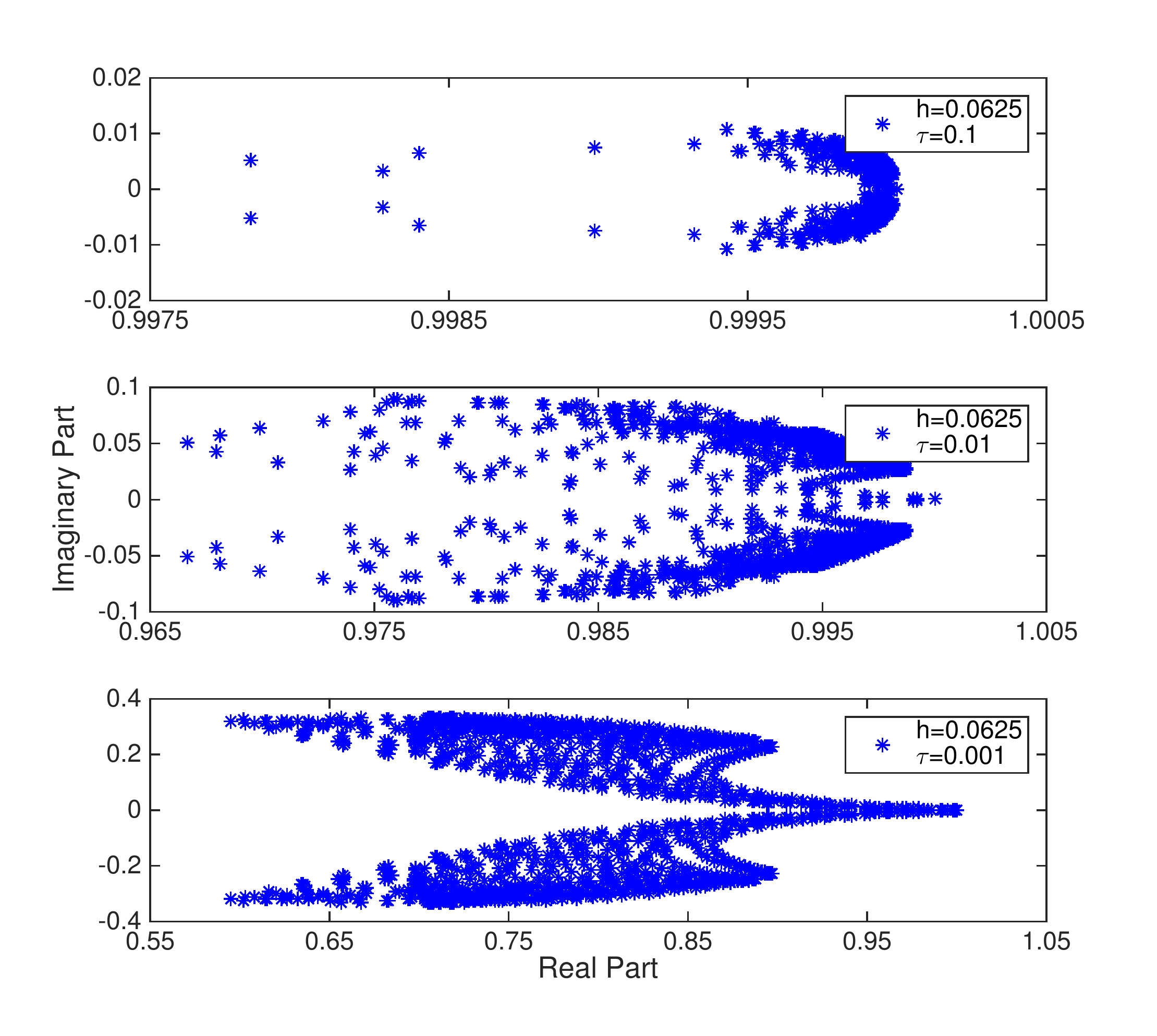}
\caption{Spectrum distribution of $\mathcal{B}^{-1}\mathcal{A}$ when varying $h$ (left) and $\tau$ (right)}
\label{spectrum4E}
\vspace{0cm}
\end{figure}

\begin{figure}
\centering
\includegraphics[width=7.8cm]{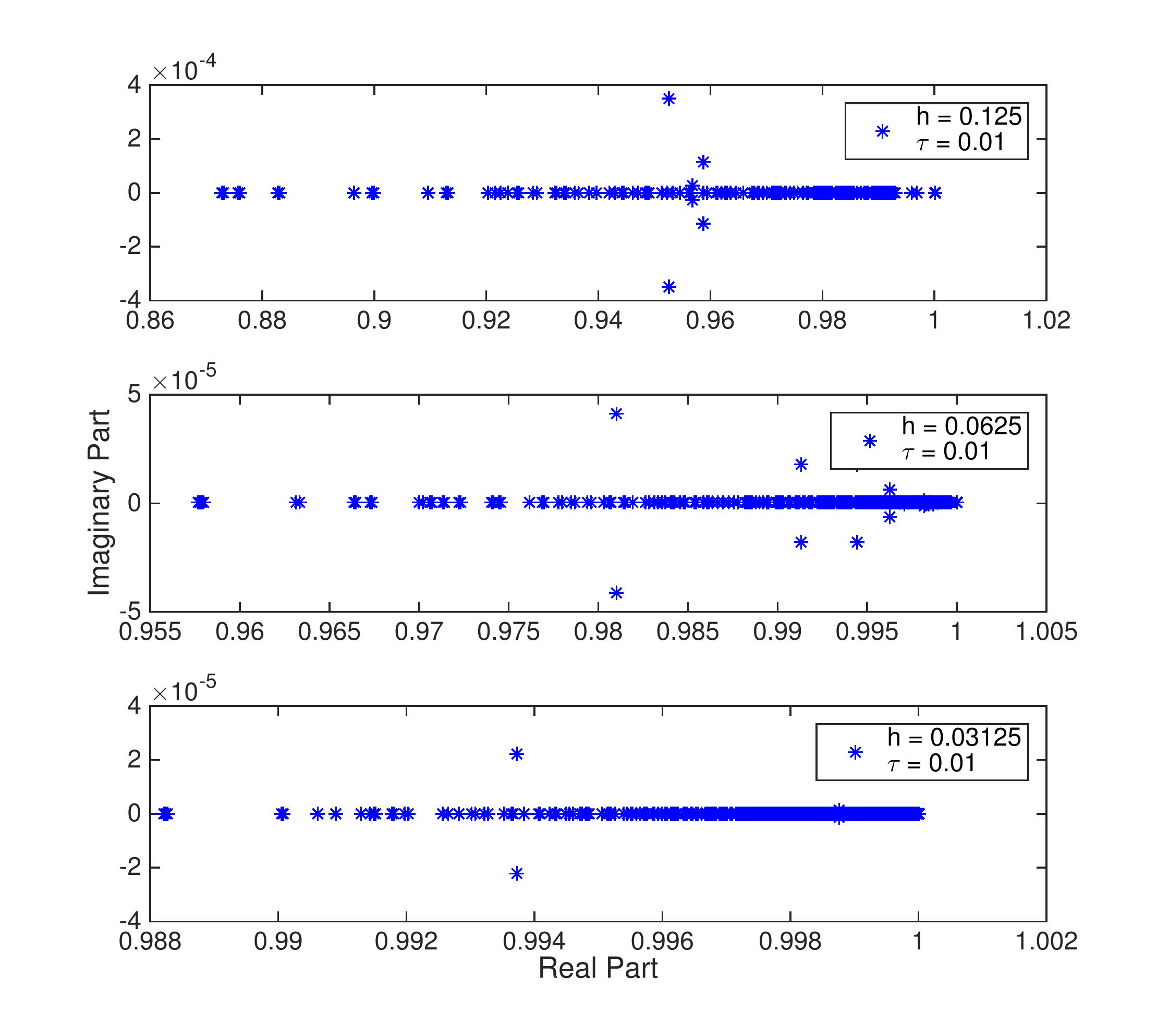}
\hspace{-0.3in}
\includegraphics[width=7.8cm]{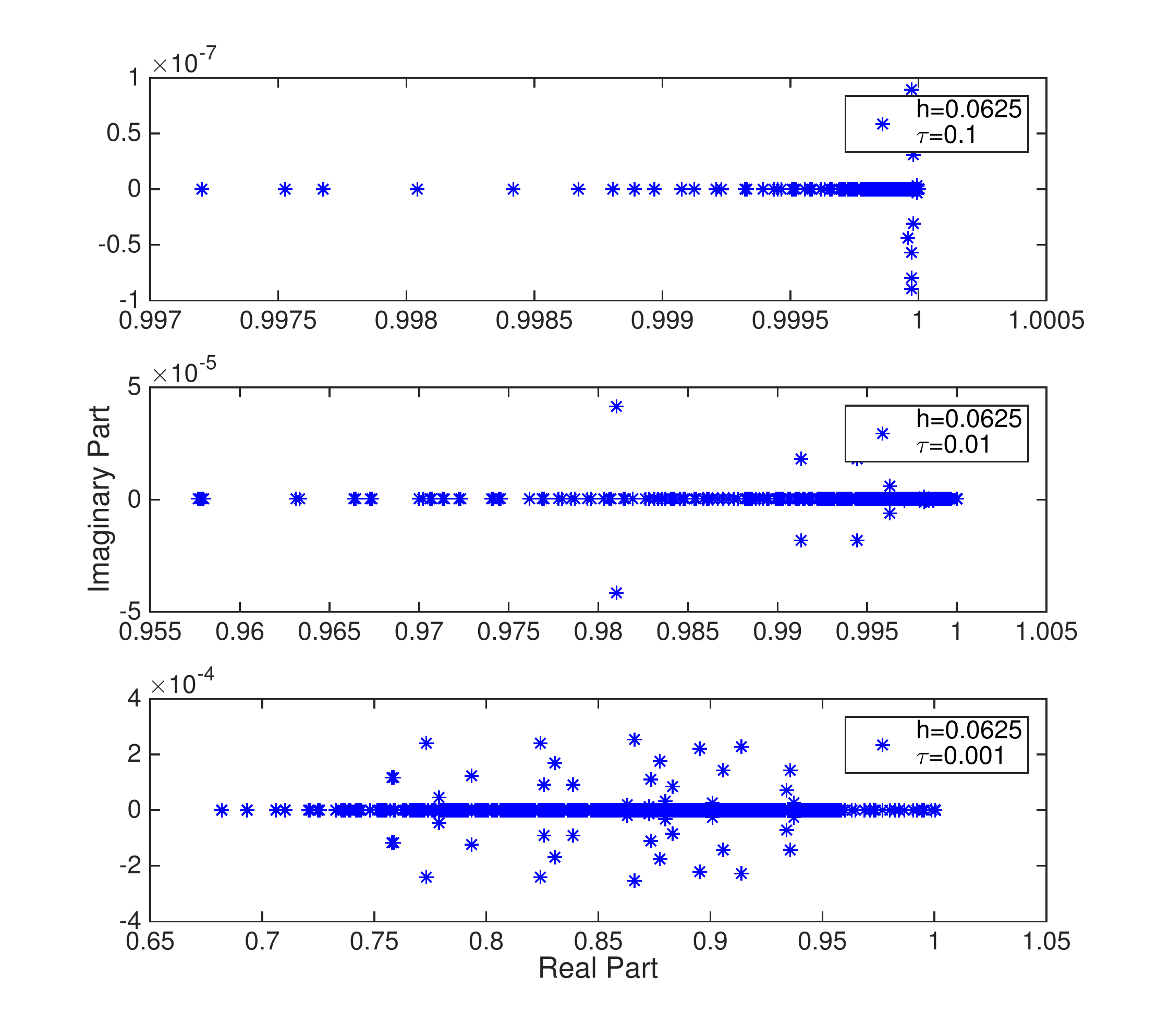}
\caption{Spectrum distribution of $\tilde{\mathcal{B}}^{-1}\mathcal{A}$ when varying $h$ (left) and $\tau$ (right)}
\label{spectrum4tBA}
\vspace{0cm}
\end{figure}

\section{Conclusions}
In this work, we propose mass lumping preconditioners for use with GMRes method to solve a class of two-by-two block linear systems which correspond to some discrete fourth order parabolic equations. Using consistent matrix in the discretization has the advantage of better solution accuracy as compared with applying mass lumping directly in the discretization. We propose to use lumped-mass matrix $\bar{M}$ and optimal multigrid algorithm with either coupled or decoupled smoother to construct practical preconditioners which are shown to be computationally very efficient. Numerical experiments indicate that preconditioned GMRes method converges uniformly with respect to both $h$ and $\tau$ and is robust for problems with degenerate diffusion coefficient. We provide theoretical analysis about the spectrum bound of the preconditioned system and estimate the convergence rate of the preconditioned GMRes method.

\subsection*{Acknowledgements} 
B. Zheng would like to acknowledge the support by NSF grant DMS-0807811 and a Laboratory Directed Research and Development (LDRD) Program from Pacific Northwest National Laboratory. L.P. Chen was supported by the National Natural Science Foundation of China under Grant No. 11501473. L. Chen was supported by NSF grant DMS-1418934 and in part by NIH grant P50GM76516. R. H. Nochetto was supported by NSF under Grants DMS-1109325 and DMS-1411808. J. Xu was supported by NSF grant DMS-1522615 and in part by US Department of Energy Grant DE-SC0014400. Computations were performed using the computational resources of Pacific Northwest National Laboratory (PNNL) Institutional Computing cluster systems. The PNNL is operated by Battelle for the US Department of Energy under Contract DE-AC05-76RL01830.

\bibliographystyle{plain}
\bibliography{Library}
\end{document}